\documentclass[11pt, a4paper, english]{amsart}

\usepackage{amsmath,amssymb,amsthm}
\usepackage{bbm}
\usepackage[cal=euler,scr=rsfs]{mathalfa}
\usepackage[all]{xy}
\usepackage[dvipdfmx,colorlinks=true,backref=page]{hyperref}
\usepackage{cleveref}

\numberwithin{equation}{section}

\theoremstyle{theorem}
\newtheorem{theorem}{Theorem}[section]
\newtheorem{proposition}[theorem]{Proposition}
\newtheorem{lemma}[theorem]{Lemma}
\newtheorem{corollary}[theorem]{Corollary}

\newtheorem{question}[theorem]{Question}

\theoremstyle{definition}
\newtheorem{definition}[theorem]{Definition}
\newtheorem{example}[theorem]{Example}

\theoremstyle{remark}

\newcommand{\N}{\mathbb{N}}
\newcommand{\Z}{\mathbb{Z}}

\newcommand{\R}{\mathbb{R}}

\newcommand{\Tr}{\operatorname{Tr}}
\newcommand{\vol}{\operatorname{vol}}
\newcommand{\Crit}{\operatorname{Crit}}
\newcommand{\Gau}{\operatorname{Gau}}

\SelectTips{cm}{}

\title{Morse inequalities for noncompact manifolds}

\author[T. Kato]{Tsuyoshi Kato}
\address{Department of Mathematics, Kyoto University, Kyoto 606-8502, Japan}
\email{tkato@math.kyoto-u.ac.jp}

\author[D. Kishimoto]{Daisuke Kishimoto}
\address{Faculty of Mathematics, Kyushu University, Fukuoka 819-0395, Japan}
\email{kishimoto@math.kyushu-u.ac.jp}

\author[M. Tsutaya]{Mitsunobu Tsutaya}
\address{Faculty of Mathematics, Kyushu University, Fukuoka 819-0395, Japan}
\email{tsutaya@math.kyushu-u.ac.jp}

\date{\today}



\begin{document}

\maketitle

\begin{abstract}
  We establish Morse inequalities for a noncompact manifold with a cocompact and properly discontinuous action of a discrete group, where Morse functions are not necessarily invariant under the group action. The inequalities are given in terms of the $L^2$-Betti numbers and functions on the acting group which describe rough configurations of critical points of a Morse function.
\end{abstract}

%%%%% Section 1 %%%%%

\section{Introduction}\label{Introduction}

Morse inequalities relate the number of critical points of a Morse function on a compact manifold to its Betti numbers. Here, compactness of a manifold is essential as it guarantees finiteness of the number of critical points and the Betti numbers. Then it is challenging to generalize Morse inequalities to noncompact manifolds. There should have been attempts for it, but there is no result except for the very special cases that critical points are finitely many \cite{DY} and have strong symmetry \cite{NS}.

Throughout the paper, let $M$ be a connected $n$-dimensional manifold without boundary, possibly noncompact, equipped with a cocompact and properly discontinuous action of a discrete group $G$ such that the orbit manifold $M/G$ is oriented. Namely, we will consider a Galois covering $G\to M\to M/G$ such that $M$ is connected and $M/G$ is a closed oriented $n$-dimensional manifold. We fix any $G$-invariant metric on $M$. If $G$ is amenable, then a manifold $M$ is a typical example of a manifold having bounded geometry and a regular exhaustion, for which the index theory is developed in \cite{R1,R2}. The purpose of this paper is to establish Morse inequalities for a certain Morse function on $M$ which is essentially irrelevant to the action of $G$.

Roe \cite{R1} crudely classified generalizations of the index theorem to noncompact manifolds into three types. In particular, for the type II theorems, including Atiyah's $\Gamma$-index theorem \cite{A} and Roe's index theorem \cite{R1,R2}, the index is interpreted through some kind of averaging or renormalization procedure. Our approach takes such procedure. To generalize Morse inequalities to noncompact manifolds, we need to replace Betti numbers by other invariants because the Betti numbers of a noncompcat manifold may not be defined. For a manifold $M$, there is a nice variant of Betti numbers, \emph{$L^2$-Betti numbers}, and we employ them for our purpose. As in \cite[Example 1.37]{L}, $L^2$-Betti numbers may be thought of as a kind of average of Betti numbers, so they should match our approach. On the other hand, Gromov \cite{G} proved that vanishing of $L^2$-Betti numbers is a quasi-isometry invariant, and as in \cite{P}, for a large class of groups, $L^2$-Betti numbers themselves are proved to be quasi-isometry invariants. Here, quasi-isometry invariants are the object of study in large scale geometry. However, we only use some results in large scale geometry, and do not work in it because quasi-isometries identify a compact set with a single point so that the number of critical points does not make sense.

Let us introduce Morse functions that we are going to consider. It is shown in \cite{H} that every connected noncompact manifold admits a function without critical points. To exclude such a trivial situation, we assume the following boundedness of Morse functions (cf. \cite{KKT}).

\begin{definition}
  A smooth function $f\colon M\to\R$ is \emph{bounded} if $\nabla^if$ is bounded for any $0\le i\le\max\{3,\lceil\frac{n}{2}\rceil\}$.
\end{definition}

Actually, boundedness of the gradient vector field $\nabla f$ is essential, and that of $f$ and its higher derivatives may be a technical condition. We note that the boundedness of a function is independent of the choice of a $G$-invariant metric on $M$ because $G$-invariant metrics on $M$ are mutually equivalent. By \eqref{c_k} below, we describe a rough configuration of critical points of a Morse function on $M$ by a function on $G$, which may be thought of as counting of critical points. In order to make this function controllable, we assume the following uniformness of Morse functions. We fix a $G$-invariant triangulation of $M$. Then as in Section \ref{Piecewise trace}, we can construct a fundamental domain $K$ such that the closure of $K$ is a finite subcomplex of $M$ and it satisfies
\begin{equation}
  \label{tiling}
  M=\coprod_{g\in G}gK.
\end{equation}
Here, we note that the closure of $K$ is a fundamental domain in the sense of \cite{KKT}.

\begin{definition}
  A Morse function $f\colon M\to\R$ is \emph{uniform} if there is $\epsilon>0$ such that any two critical points of $f$ are at least $2\epsilon$ distant.
\end{definition}

In Sectioin \ref{Morse inequalities}, we define a \emph{strongly uniform} Morse function on $M$ as a uniform Morse function on $M$ satisfying additional mild conditions on critical points. Let $f\colon M\to\R$ be a Morse function, and let $\Crit_k(f)$ denote the set of critical points of $f$ with index $k$. Define a function
\begin{equation}
  \label{c_k}
  c_k\colon G\to\R,\quad g\mapsto|\Crit_k(f)\cap gK|.
\end{equation}
Then this is a function describing a rough configuration of critical points of $f$ with index $k$, mentioned above. This way of counting of discrete points on $M$ was discovered in \cite{KKT}. Observe that $c_k$ is a bounded function for each $k\ge 0$ whenever $f$ is uniform, because the closure of a fundamental domain $K$ is compact.

We set notation. Let $b_k^{(2)}$ denote the $k$-th $L^2$-Betti number of $M$. Let $\ell^\infty(G)$ denote the Banach space of bounded functions on $G$. Then $G$ acts on $\ell^\infty(G)$ by $(g\cdot\phi)(x)=\phi(xg)$ for $g,x\in G$ and $\phi\in\ell^\infty(G)$. Let $\mathcal{I}$ denote the subspace of $\ell^\infty(G)$ generated by $\phi-g\cdot\phi$ for $g\in G$ and $\phi\in\ell^\infty(G)$, and let $\bar{\mathcal{I}}$ denote the closure of $\mathcal{I}$ by the weak topology. Taking modulo $\bar{\mathcal{I}}$ may be thought of as averaging process mentioned above. For $\phi_1,\phi_2\in\ell^\infty(G)$, we write
\begin{itemize}
  \item $\phi_1\ge\phi_2\mod\bar{\mathcal{I}}$ if there is $\phi_3\in\bar{\mathcal{I}}$ such that the pointwise inequality $\phi_1\ge\phi_2+\phi_3$ holds, and

  \item $\phi_1\approx\phi_2\mod\bar{\mathcal{I}}$ if $\phi_1\ge\phi_2\mod\bar{\mathcal{I}}$ and $\phi_2\ge\phi_1\mod\bar{\mathcal{I}}$.
\end{itemize}
The first relation is a preorder on $\ell^\infty(G)$, not necessarily a partial order, so that we need the second relation. The quotient $\ell^\infty(G)/\mathcal{I}$ is isomorphic to the $0$-th uniformly finite homology of $M$ defined by Block and Weinberger \cite{BW} (see also \cite{BNW}), and is related with finite propagation operators that play an important role in this paper \cite{KKT1,KKT2,KKT3}. Let $\mathbbm{1}\in\ell^{\infty}(G)$ denote the constant function on $G$ with values $1$.

Now we state the main theorem.

\begin{theorem}
  \label{main}
  For a strongly uniform bounded Morse function $f\colon M\to\R$, we have
  \[
  c_k-c_{k-1}+\cdots+(-1)^kc_0\ge(b_k^{(2)}-b_{k-1}^{(2)}+\cdots+(-1)^kb_0^{(2)})\mathbbm{1}\mod\bar{\mathcal{I}}
  \]
  for $k=0,1,\ldots,n-1$, and
  \[
    c_n-c_{n-1}+\cdots+(-1)^nc_0\approx(-1)^n\chi(M/G)\mathbbm{1}\mod\bar{\mathcal{I}}.
  \]
\end{theorem}

We remark that Theorem \ref{main} makes sense only when $G$ is amenable because $\ell^\infty(G)\ne\bar{\mathcal{I}}$ if and only if $G$ is amenable (Proposition \ref{amenability}). By definition, $G$ is amenable if there is a $G$-invariant positive linear map $\mu\colon\ell^\infty(G)\to\R$ with $\mu(\mathbbm{1})=1$, which is called a $G$-invariant \emph{mean}. As a corollary to Theorem \ref{main}, we have the following mean value Morse inequalities.

\begin{corollary}
  \label{mean value}
  Let $f\colon M\to\R$ be a strongly uniform bounded Morse function. If $G$ is amenable, then for any $G$-invariant mean $\mu\colon\ell^\infty(G)\to\R$,
  \[
    \mu(c_k)-\mu(c_{k-1})+\cdots+(-1)^k\mu(c_0)\ge b_k^{(2)}-b_{k-1}^{(2)}+\cdots+(-1)^kb_0^{(2)}
  \]
  for $k=0,1,\ldots,n-1$ and
  \[
    \mu(c_n)-\mu(c_{n-1})+\cdots+(-1)^n\mu(c_0)=(-1)^n\chi(M/G).
  \]
\end{corollary}

We can deduce from Theorem \ref{main} the weak Morse inequalities
\[
  c_k\ge b_k^{(2)}\mod\bar{\mathcal{I}}
\]
for $k=0,1,\ldots,n$, and by applying a property of the module of coinvariants $\ell^\infty(G)/\mathcal{I}$ proved in \cite{KKT} to it, we get the following property of the $L^2$-Betti numbers, which has its own interest.

\begin{corollary}
  \label{betti critical points}
  Let $f\colon M\to\R$ be a strongly uniform bounded Morse function. If $G$ is an infinite amenable group and $b^{(2)}_k\ne 0$, then $f$ must have infinitely many critical points with index $k$.
\end{corollary}

We also see that Theorem \ref{main} recovers the result of Novikov and Shubin \cite{NS} on Morse inequalities for a Morse function on $M/G$ and the $L^2$-Betti numbers of $M$ when $G$ is amenable (Corollary \ref{Novikov-Shubin}).

We prove Theorem \ref{main} by employing the Witten deformation by a Morse function. Then we need to disassemble the trace of an operator on the Hilbert space of $L^2$-sections of a vector bundle over $M$ with respect to the cover \eqref{tiling}. In Section \ref{Piecewise trace}, we introduce a piecewise trace-class operator and its piecewise trace which is the desired disassembly of the usual (or Roe's) trace. In Section \ref{Operator with Gaussian propagation}, we introduce an operator whose kernel function has propagation subordinate to a Gaussian function centered at the diagonal set, which we call an operator with Gaussian propagation. We prove the trace property of the piecewise trace of a product of operators with Gaussian propagation. We also consider the functional calculus of a generalized Dirac operators such that the resulting operator has Gaussian propagation. In Section \ref{Piecewise normalized Betti number}, we introduce a disassembly of the normalized Betti number of Roe \cite{R2,R3}, which is related to the $L^2$-Betti number of $M$, with respect to the cover \eqref{tiling}, and prove its invariance under a certain deformation, including the Witten deformation by a bounded function. In Section \ref{Morse inequalities}, we collect all results obtained so far, and apply the Witten deformation to prove Theorem \ref{main}. Some conditions on a Morse function that we assume seem to be quite technical. In Section \ref{Questions}, we pose some questions on such technical conditions.

\subsection*{Acknowledgement}

The authors are partially supported by JSPS KAKENHI Grant numbers JP23K22394 (Kato), JP22K03284 (Kishimoto), and JP22K03317 (Tsutaya)

%%%%% Section 2 %%%%%

\section{Piecewise trace}\label{Piecewise trace}

In \cite{R1}, Roe defined the trace of a uniform operator over a noncompact manifold of bounded geometry with regular exhaustion, and use it to build the index theory for such a noncompact manifold. In this section, we introduce the piecewise trace of an operator on the space of $L^2$-sections of a $G$-invariant vector bundle over $M$, which may be thought of as disassembly of Roe's trace with respect to the action of $G$.

%%%%% Subsection 2.1 %%%%%

\subsection{Fundamental domain}

First of all, we construct a fundamental domain. Choose a $G$-invariant triangulation $L$ of $M$, that is, the lift of a triangulation of $M/G$. Then $L/G$ is a triangulation of $M/G$. For each open $n$-simplex of $L/G$, we choose one lift to $L$. For each open $k$-simplex with $k<n$, we choose one lift to $M$ which is in the closure of some open $n$-simplex of $M$ that we have chosen. Now we define a \emph{fundamental domain} $K$ as the union of those open simplices of $M$. Observe that the closure of $K$ is the union of the closure of open $n$-simplices in $K$. Then the closure of $K$ is a finite complex, and a fundamental domain in the sense of \cite{KKT}. Remark that a fundamental domain $K$ needs not be connected. By definition, we have
\[
  M=\coprod_{g\in G}gK.
\]
In the sequel, we fix a fundamental domain $K$.

%%%%% Subsection 2.2 %%%%%

\subsection{Piecewise Hilbert-Schmidt operator}

Let $E\to M$ be a $G$-invariant vector bundle with $G$-invariant metric, that is, $E$ is the lift of a vector bundle over $M/G$ with metric. Let $L^2(E)$ denote the Hilbert space of $L^2$-sections of $E$. Then the map $g^{-1}\colon gK\to K$ induces an isometry $(g^{-1})^*\colon L^2(E\vert_K)\to L^2(E\vert_{gK})$. We choose an orthonormal basis $\{e_i\mid i\in I\}$ of $L^2(E\vert_K)$, and set $e_i^g=(g^{-1})^*(e_i)$ for $g\in G$. Then $\{e_i^g\mid i\in I\}$ is an orthonormal basis of $L^2(E\vert_{gK})$. Moreover, if we consider $e_i^g$ as an $L^2$-section of the entire vector bundle $E$ in the obvious way, then $\{e_i^g\mid i\in I,\,g\in G\}$ is an orthonormal basis of $L^2(E)$. For $u\in L^2(E)$, let $\|u\|$ denote the $L^2$-norm.

\begin{definition}
  A bounded operator $A\colon L^2(E)\to L^2(E)$ is called \emph{piecewise Hilbert-Schmidt} if
  \[
    \rho_2(A)\colon G\to\R,\quad g\mapsto\left(\sum_{i\in I}\|Ae_i^g\|^2\right)^\frac{1}{2}.
  \]
  is a well-defined bounded function.
\end{definition}

Let $\mathbf{1}_S\colon M\to\{0,1\}$ denote the characteristic function of a subset $S\subset M$. An operator $A\colon L^2(E)\to L^2(E)$ is a piecewise Hilbert-Schmidt operator if and only if $A\mathbf{1}_{gK}$ is a Hilbert-Schmidt operator for each $g\in G$ such that $\sup_{g\in G}\|A\mathbf{1}_{gK}\|_\mathrm{HS}$ is bounded, where $\|\cdot\|_\mathrm{HS}$ denotes the Hilbert-Schmidt norm. Observe that
\begin{equation}
  \label{rho_2}
  \rho_2(A)(g)=\|A\mathbf{1}_{gK}\|_\mathrm{HS}.
\end{equation}
Then $\rho_2(A)$ is independent of the choice of an orthonormal basis of $L^2(E\vert_K)$, and we may think of $\rho_2$ as disassembly of the Hilbert-Schmidt norm with respect to the cover \eqref{tiling}.  The following lemma is immediate from the definition of $\rho_2$.

\begin{lemma}
  Let $A,B\colon L^2(E)\to L^2(E)$ be piecewise Hilbert-Schmidt operators.
  \begin{enumerate}
    \item $\rho_2(A+B)\le\rho_2(A)+\rho_2(B)$.

    \item $\rho_2(cA)=|c|\rho_2(A)$ for $c\in\R$.

    \item $\rho_2(A)=0$ if and only if $A=0$.
  \end{enumerate}
\end{lemma}

For a bounded operator $A\colon L^2(E)\to L^2(E)$, let $\|A\|$ denote its operator norm.

\begin{proposition}
  \label{HS norm}
  Let $A\colon L^2(E)\to L^2(E)$ be a bounded operator, and let $B\colon L^2(E)\to L^2(E)$ be a piecewise Hilbert-Schmidt operator. Then $AB$ is a piecewise Hilbert-Schmidt operator such that
  \[
    \rho_2(AB)\le\|A\|\rho_2(B).
  \]
\end{proposition}

\begin{proof}
  By \eqref{rho_2}, $\rho_2(AB)(g)=\|AB\mathbf{1}_{gK}\|_\mathrm{HS}\le\|A\|\|B\mathbf{1}_{gK}\|_\mathrm{HS}=\|A\|\rho_2(B)(g)$ for any $g\in G$. Then the statement follows.
\end{proof}

Let $A\colon L^2(E)\to L^2(E)$ be a Hilbert-Schmidt operator, and let $B\colon L^2(E)\to L^2(E)$ be a bounded operator. Then we have
\[
  \|AB\|_\mathrm{HS}\le\|A\|_\mathrm{HS}\|B\|\quad\text{and}\quad\|A^*\|_\mathrm{HS}=\|A\|_\mathrm{HS}.
\]
Hence we may expect that $\rho_2$ has analogous properties, but it fails. Choose any $e\in L^2(E\vert_K)$ with $\|e\|=1$ and $g\in G$ with $g\ne 1$. We define a piecewise Hilbert-Schmidt operator
\[
  A=\langle-,e\rangle(g^{-1})^*(e)
\]
where $\langle-,-\rangle$ denotes the $L^2$-inner product. Then $A^*=\langle-,(g^{-1})^*(e)\rangle e$ is a piecewise Hilbert-Schmidt operator. Observe that $\rho_2(A^*A)(1)=1\ge 0=\rho_2(A^*)(1)\|A\|$ and $\rho_2(A)(1)=1\ne 0=\rho_2(A^*)(1)$. Then
\[
  \rho_2(A^*A)\not\le\rho_2(A^*)\|A\|\quad\text{and}\quad\rho_2(A^*)\ne\rho_2(A).
\]

Let $A\colon L^2(E)\to L^2(E)$ be an operator. We say that $A$ is represented by a \emph{kernel function} $k_A$ if $k_A$ is a section of the vector bundle $E^*\boxtimes E\to M\times M$ such that for any $u\in L^2(E)$,
\[
  (Au)(x)=\int_Mk_A(x,y)u(y)\vol(y)
\]
where $\vol$ denotes the pullback of the volume form of $M/G$ to $M$. We also say that an operator $A$ is \emph{smoothing} if it is represented by a smooth kernel function.

\begin{lemma}
  \label{HS norm integral}
  If $A\colon L^2(E)\to L^2(E)$ is a piecewise Hilbert-Schmidt smoothing operator, then for each $g\in G$,
  \[
    \rho_2(A)(g)=\left(\int_M\int_{gK}|k_A(x,y)|^2\vol(y)\vol(x)\right)^\frac{1}{2}.
  \]
\end{lemma}

\begin{proof}
  Observe that
  \begin{align*}
    (\rho_2(A)(g))^2&=\sum_{i\in I}\left(\|Ae_i^g\|^2\right)^{\frac{1}{2}}\\
    &=\sum_{i\in I}\int_M\left|\int_Mk_A(x,y)e_i^g(y)\vol(y)\right|^2\vol(x)\\
    &=\int_M\sum_{\substack{i\in I\\h\in G}}\left|\int_Mk_A(x,y)\mathbf{1}_{gK}e_i^h(y)\vol(y)\right|^2\vol(x)\\
    &=\int_M\int_M|k_A(x,y)\mathbf{1}_{gK}|^2\vol(y)\vol(x)\\
    &=\int_M\int_{gK}|k_A(x,y)|^2\vol(y)\vol(x).
  \end{align*}
  where we use Parseval's identity for the fourth equality. Then the statement follows.
\end{proof}

\begin{lemma}
  \label{HS inner product}
  If $A,B\colon L^2(E)\to L^2(E)$ are piecewise Hilbert-Schmidt operators, then
  \[
    \rho_2(A,B)\colon G\to\R,\quad g\mapsto\sum_{i\in I}\langle Ae_i^g,Be_i^g\rangle.
  \]
  is a well-defined bounded function which is independent of the choice of an orthonormal basis of $L^2(E\vert_K)$.
\end{lemma}

\begin{proof}
  Define an operator
  \begin{equation}
    \label{sign}
    P=\sum_{i\in I,\,g\in G}\mathrm{sgn}(\langle Ae_i^g,Be_i^g\rangle)\langle-,e_i^g\rangle e_i^g
  \end{equation}
  where we put $\mathrm{sgn}(0)=0$. Then $BP$ is a piecewise Hilbert-Schmidt operator such that $\rho_2(BP)=\rho_2(B)$ and $\rho_2(A,BP)=\sum_{i\in I}|\langle Ae_i^g,Be_i^g\rangle|$. Observe that for any piecewise Hilbert-Schmidt operators $C$ and $D$, the Schwartz inequality
  \begin{equation}
    \label{Schwartz inequality}
    \rho_2(C,D)\le\rho_2(C)\rho_2(D)
  \end{equation}
  holds. Thus $\rho_2(A,BP)$ is a bounded function, and therefore $\rho_2(A,B)$ is a well-defined bounded function. The independence of the choice of an orthonormal basis of $L^2(E\vert_K)$ follows from the fact that
  \[
    \rho_2(A,B)(g)=\langle A\mathbf{1}_{gK},B\mathbf{1}_{gK}\rangle
  \]
  for $g\in G$.
\end{proof}

%%%%% Subsection 2.3 %%%%%

\subsection{Piecewise trace-class operator}

We define a piecewise trace-class operator as an analogy to a trace-class operator.

\begin{definition}
  An operator $A\colon L^2(E)\to L^2(E)$ is a \emph{piecewise trace-class} operator if $A=B^*C$ for some piecewise Hilbert-Schmidt operators $B,C\colon L^2(E)\to L^2(E)$. Its \emph{piecewise trace} is defined by
  \[
    \Tr(A)=\rho_2(B,C).
  \]
\end{definition}

By Lemma \ref{HS inner product}, the piecewise trace is a well-defined bounded function on $G$ which is independent of the choice of an orthonormal basis of $L^2(E\vert_K)$. It is easy to see that $\Tr(A)$ is independent of the choice of a decomposition $A=B^*C$. We show the basic properties of a piecewise trace.

\begin{lemma}
  \label{trace integral}
  If $A=B^*C$ for piecewise Hilbert-Schmidt smoothing operators $B,C\colon L^2(E)\to L^2(E)$, then $A$ is a trace-class smoothing operator such that
  \[
    \Tr(A)(g)=\int_{gK}\mathrm{tr}(k_A(x,x))\vol(x).
  \]
\end{lemma}

\begin{proof}
  By definition, $A$ is of piecewise trace-class, and since the composite of smoothing operators is smoothing, $A$ is a smoothing operator. Since $\Tr(A)(g)$ is the trace of the trace-class operator $\mathbf{1}_gA\mathbf{1}_g$, the equality in the statement follows quite similarly to
  \cite[Theorems 8.12]{R3}.
\end{proof}

\begin{proposition}
  Let $A\colon L^2(E)\to L^2(E)$ be a piecewise trace-class smoothing operator. Then the piecewise trace of $A$ modulo $\mathcal{I}$ is independent of the choice of a fundamental domain $K$.
\end{proposition}

\begin{proof}
  By Lemma \ref{trace integral}, $\Tr(A)$ is given by the integral of the bounded $n$-form $\mathrm{tr}(k_A(x,x))\vol(x)$ in the sense of \cite{KKT}. It is shown in \cite{M} that such an integral modulo $\mathcal{I}$ is independent of the choice of $K$.
\end{proof}

For a piecewise trace-class operator $A\colon L^2(E)\to L^2(E)$, we can define a bounded function
\[
  \rho_1(A)\colon G\to\R,\quad g\mapsto\sum_{i\in I}|\langle e_i^g,Ae_i^g\rangle|.
\]
Indeed, the proof of Lemma \ref{HS inner product} implies that $\rho_1(A)$ is a well-defined bounded function. Here, we remark that $\rho_1$ depends on the choice of an orthonormal basis of $L^2(E\vert_K)$, so when we consider $\rho_1$, we always fix an orthonormal basis of $L^2(E\vert_K)$. We record properties of $\rho_1$ which are immediate from the definition.

\begin{lemma}
  Let $A,B\colon L^2(E)\to L^2(E)$ be piecewise trace-class operators.
  \begin{enumerate}
    \item $\rho_1(A+B)\le\rho_1(A)+\rho_1(B)$.

    \item $\rho_1(cA)=|c|\rho_1(A)$ for $c\in\R$.
  \end{enumerate}
\end{lemma}

\begin{proposition}
  \label{trace positive}
  Let $A\colon L^2(E)\to L^2(E)$ be a piecewise trace-class operator. If $A$ is positive, then
  \[
    \Tr(A)=\rho_1(A)\ge 0.
  \]
\end{proposition}

\begin{proof}
  If $A$ is positive, then $\langle e_i^g,Ae_i^g\rangle\ge 0$ for any $g\in G$ and $i\in I$, implying $\Tr(A)=\rho_1(A)\ge 0$.
\end{proof}

\begin{proposition}
  \label{HS trace-class}
  If $A,B\colon L^2(E)\to L^2(E)$ are piecewise Hilbert-Schmidt operators, then
  \[
    \rho_1(A^*B)\le\rho_2(A)\rho_2(B).
  \]
\end{proposition}

\begin{proof}
  Let $P$ be as in \eqref{sign}. Then $\rho_1(A^*B)=\rho_2(A,BP)$, so by \eqref{Schwartz inequality}, $\rho_1(A^*B)\le\rho_2(A)\rho_2(BP)=\rho_2(A)\rho_2(B)$.
\end{proof}

%%%%% Subsection 2.4 %%%%%

\subsection{Roe's trace}

Recall that a discrete group $G$ is \emph{amenable} if there is a $G$-invariant positive linear map $\mu\colon\ell^\infty(G)\to\R$ satisfying $\mu(\mathbbm{1})=1$, called a $G$-invariant \emph{mean}. We have the following characterization of amenability.

\begin{proposition}
  \label{amenability}
  The following conditions are equivalent.

  \begin{enumerate}
    \item $G$ is amenable.

    \item $\ell^\infty(G)\ne\bar{\mathcal{I}}$.
  \end{enumerate}
\end{proposition}

\begin{proof}
  Let $\mu\colon\ell^\infty(G)\to\R$ be a $G$-invariant mean. Then for $\phi\in\ell^\infty(G)$,
  \[
    |\mu(\phi)|\le|\mu(\sup_{x\in G}|\phi(x)|\mathbbm{1})|=\sup_{x\in G}|\phi(x)|
  \]
  and then $\mu$ is bounded, implying $\mu(\bar{\mathcal{I}})=0$. Thus by $\mu(\mathbbm{1})=1$, (1) implies (2). By \cite{BW,BNW}, if $G$ is not amenable then $\ell^\infty(G)=\mathcal{I}$. Hence (2) implies (1).
\end{proof}

We compare the piecewise trace with Roe's trace \cite{R1} when $G$ is amenable. Recall that F\o lner's theorem also characterizes amenability, which states that for any $\epsilon>0$ and a nonempty finite subset $S\subset G$, there is a finite subset $F\subset G$ satisfying
\[
  \frac{|F\triangle gF|}{|F|}<\epsilon
\]
for any $g\in S$. For $k\ge 1$, we define $S_k=\{g\in G\mid d(K,gK)\le k\}$. Then we get a sequence $S_1\subset S_2\subset\cdots$ of finite subsets of $G$ which exhausts $G$. By F\o lner's theorem, there is a finite subset $F_k$ of $G$ satisfying
\[
  \frac{|F_k\triangle gF_k|}{|F_k|}<\frac{1}{k}
\]
for any $g\in S_k$. Let $M_k$ be the union of all $gK$ for $g\in F_k$. Then as is shown in \cite[p.106]{R1}, $\{M_k\}_{k\ge 1}$ is a regular exhaustion of $M$, hence by \cite[(3.5)]{R3}, Roe's trace of a suitable smoothing operator $A\colon L^2(E)\to L^2(E)$ is given by
\[
  \tau(A)=\lim_{k\to\infty}\frac{1}{\vol(M_k)}\int_{M_k}\mathrm{tr}(k_A(x,x))\vol(x).
\]
On the other hand, we can define a $G$-invariant finitely additive probability measure on $G$ by
\[
  \mu(S)=\lim_\omega\frac{|F_k\cap S|}{|F_k|}
\]
for $S\subset G$, where $\omega$ is a nonprincipal ultrafilter on $\N$, which yields a $G$-invariant mean $\mu\colon\ell^\infty(G)\to\R$. Thus we can easily deduce the following, where $\vol(K)=\vol(M/G)$.

\begin{proposition}
  \label{disassmebly}
  Let $A\colon L^2(E)\to L^2(E)$ be a piecewise trace-class operator, and let $\mu\colon\ell^\infty(G)\to\R$ be the above $G$-invariant mean. Then
  \[
    \mu(\Tr(A))=\frac{\tau(A)}{\vol(M/G)}.
  \]
\end{proposition}

%%%%% Section 3 %%%%%

\section{Operator with Gaussian propagation}\label{Operator with Gaussian propagation}

In this section, we define an operator with Gaussian propagation, and show the trace property of the piecewise trace of the product of operators with Gaussian propagation. We also investigate functional calculus of a generalized Dirac operator that produces operators with Gaussian propagation.

%%%%% Subsection 3.1 %%%%%

\subsection{Trace property}

Let $E\to M$ be a $G$-invariant vector bundle with $G$-invariant metric.

\begin{definition}
  A smoothing operator $A\colon L^2(E)\to L^2(E)$ has \emph{Gaussian propagation} if there are $C_1,C_2>0$ such that
  \[
    |k_A(x,y)|\le C_1e^{-C_2d(x,y)^2}.
  \]
\end{definition}

Since the group $G$ is the quotient of the fundamental group of a compact manifold $M/G$, it is finitely generated. We choose a finite generating set of $G$, and consider the associated word metric on $G$, where those metrics are mutually quasi-isometric. Now we take any $g,h\in G$. By the fundamental observation in geometric group theory, an inclusion $G\to M$ is a quasi-isometry, so there are $C_1,C_2,C_3,C_4>0$ which are independent of the choice of $g,h$ such that for any $x\in gK,\,y\in hK$, we have
\begin{equation}
  \label{quasi-isometry}
  C_1e^{-C_2d(g,h)^2}\le e^{-d(x,y)^2}\le C_3e^{-C_4d(g,h)^2}
\end{equation}
For $m\ge 0$, let $G_m=\{g\in G\mid d(g,1)\le m\}$. Then we get a sequence of subsets of $G$
\[
  G_0\subset\cdots\subset G_m\subset G_{m+1}\subset\cdots
\]
which exhausts $G$. As in \cite{Mi}, there is $C>0$ such that
\begin{equation}
  \label{growth}
  |G_m|\le e^{Cm}.
\end{equation}

Let $\Gau(E)$ denote the set of smoothing operators on $L^2(E)$ with Gaussian propagation.

\begin{proposition}
  \label{algebra}
  The set of operators $\Gau(E)$ is a $*$-algebra.
\end{proposition}

\begin{proof}
  If $A$ is a smoothing operator, then $A^*$ is also a smoothing operator with $k_{A^*}(x,y)=k_A(y,x)^*$. Hence $\Gau(E)$ is closed under $*$. Clearly, $\Gau(E)$ is closed under linear combinations, so it remains to show that $\Gau(E)$ is closed under products. Let $A,B\in\Gau(E)$, and let $g_1,g_2\in G$ with $d(g_1,g_2)=l$. Take any $x_1\in g_1K$ and $x_2\in g_2K$. Then
  \begin{align*}
    |k_{AB}(x_1,x_2)|&\le\int_M|k_A(x_1,y)k_B(y,x_2)|\vol(y)\\
    &\le\sum_{m=0}^\infty\sum_{g\in G_m}\int_{gg_1K}|k_A(x_1,y)k_B(y,x_2)|\vol(y)\\
    &\le\sum_{m=0}^\infty\sum_{g\in G_m}\int_{gg_1K}C_1e^{-C_2(d(x_1,y)^2+d(y,x_2)^2)}\vol(y)
  \end{align*}
  for some $C_1,C_2>0$ which are independent of the choice of $g_1,g_2,x_1,x_2$. By \eqref{quasi-isometry} and \eqref{growth},
  \begin{align*}
    &\sum_{m=0}^\infty\sum_{g\in G_m}\int_{gg_1K}C_1e^{-C_2(d(x_1,y)^2+d(y,x_2)^2)}\vol(y)\\
    &\le\sum_{m=0}^\infty\sum_{g\in G_m}C_1e^{-C_3(m^2+(m-l)^2)}\\
    &\le e^{-C_3l^2}\sum_{m=0}^\infty C_1e^{-2C_3(m^2-lm)+C_4m}\\
    &\le C_5e^{-C_3l^2}
  \end{align*}
  for some $C_3,C_4,C_5>0$ which are independent of the choice of $g_1,g_2,x_1,x_2$. Moreover, by \eqref{quasi-isometry},
  \[
    C_5e^{-C_3l^2}\le C_6e^{-C_7d(x_1,x_2)^2}.
  \]
  for some $C_6,C_7>0$ which are independent of the choice of $g_1,g_2,x_1,x_2$. Then we obtain $|k_{AB}(x_1,x_2)|\le C_6e^{-C_7d(x_1,x_2)^2}$. Thus $AB$ belongs to $\Gau(E)$, completing the proof.
\end{proof}

\begin{proposition}
  \label{Gaussian HS}
   Elements of $\Gau(E)$ are piecewise Hilbert-Schmidt.
\end{proposition}

\begin{proof}
  The proof of Lemma \ref{HS norm integral} implies that a smoothing operator $A\colon L^2(E)\to L^2(E)$ is piecewise Hilbert-Schmidt if and only if
  \begin{equation}
    \label{HS bounded}
    \sup_{g\in G}\int_M\int_{gK}|k_A(x,y)|^2\vol(y)\vol(x)<\infty.
  \end{equation}
   Now we let $A\in\Gau(E)$, and take any $h\in G$ and $y\in hK$. By \eqref{quasi-isometry},
  \[
    \int_M|k_A(x,y)|^2\vol(x)=\sum_{g\in G}\int_{gK}|k_A(x,y)|^2\vol(x)\le\sum_{g\in G}C_1e^{-C_2d(g,h)^2}
  \]
  for some $C_1,C_2>0$ which are independent of the choice of $h$ and $y$. By \eqref{growth}, we have
  \[
    \sum_{g\in G}C_1e^{-C_2d(g,h)^2}\le\sum_{m=0}^\infty\sum_{g\in G_m}C_1e^{-C_2d(gh,h)^2}\le\sum_{m=0}^\infty C_1e^{-C_2m^2+C_3m}<C_4.
  \]
  for some $C_3,C_4>0$ which are independent of the choice of $h$ and $y$. Then \eqref{HS bounded} holds, hence $A$ is piecewise Hilbert-Schmidt as stated.
\end{proof}

\begin{corollary}
  \label{Gaussian trace-class}
  If $A,B\in\Gau(E)$, then $AB$ is of piecewise trace-class.
\end{corollary}

\begin{proof}
  If $A,B\in\Gau(E)$, then by Propositions \ref{algebra} and \ref{Gaussian HS}, $A^*$ and $B$ are piecewise Hilbert-Schmidt operators. Hence $AB=(A^*)^*B$ is of piecewise trace-class.
\end{proof}

We consider a function $F\colon G\times G\to\R$ such that there are $C_1,C_2>0$ satisfying
\begin{equation}
  \label{propagation F}
  |F(g,h)|\le C_1e^{-C_2d(g,h)^2}
\end{equation}
for all $g,h\in G$.

\begin{lemma}
  \label{sum F}
  There is $C>0$ such that for any $g\in G$, we have
  \[
    \sum_{h\in G}|F(h,g)|\le C\quad\text{and}\quad\sum_{h\in G}|F(g,h)|\le C.
  \]
\end{lemma}

\begin{proof}
  By \eqref{growth}, there is $C_3>0$ such that for any $m\ge 0$ and $g\in G$,
  \[
    \sum_{h\in G_m}|F(hg,g)|=\sum_{k=0}^m\sum_{h\in G_k-G_{k-1}}|F(hg,g)|\le\sum_{k=0}^mC_1e^{-C_2k^2+C_3k}.
  \]
  Clearly, the last term is bounded, and then by taking $m\to\infty$, the first inequality is obtained. By symmetry, the second inequality is also obtained.
\end{proof}

By Lemma \ref{sum F}, we can define bounded functions $F_1,F_2\colon G\to\R$ by
\[
  F_1(g)=\sum_{h\in G}F(h,g)\quad\text{and}\quad F_2(g)=\sum_{h\in G}F(g,h).
\]
Let $\widehat{\mathcal{I}}$ denote the closure of $\mathcal{I}$ by the $\ell^\infty$-norm topology. Note that $\widehat{\mathcal{I}}\subset\bar{\mathcal{I}}$.

\begin{lemma}
  \label{F_1=F_2}
  $F_1\equiv F_2\mod\widehat{\mathcal{I}}$.
\end{lemma}

\begin{proof}
  For $m\ge 0$ and $g\in G$, let
  \[
    F^m_1(g)=\sum_{h\in G_m}F(g,hg)\quad\text{and}\quad F_2^m(g)=\sum_{h\in G_m}F(h^{-1}g,g).
  \]
  By Lemma \ref{sum F}, $F_1^m$ and $F_2^m$ are bounded functions on $G$. By \eqref{growth},
  \[
    |F_1(g)-F_1^m(g)|\le\sum_{k=m}^\infty\sum_{h\in G_k-G_{k-1}}|F(g,hg)|\le\sum_{k=m}^\infty C_1e^{-C_2k^2+C_3k}\le C_4e^{-C_5m^2}
  \]
  for some $C_3,C_4,C_5>0$ which are independet of the choice of $m$ and $g$. Then $F_1^m$ converges uniformly to $F_1$ as $m\to\infty$. By symmetry, $F_2^m$ also converges uniformly to $F_2$ as $m\to\infty$. For $g\in G$, we define a function $\phi_g\colon G\to\R$ by $\phi_g(h)=F(g,gh)$. By \eqref{propagation F}, $\phi_g$ is bounded. We set
  \[
    \Phi_m(h)=\sum_{g\in G_m}(1-h^{-1})\phi_g(h).
  \]
  Then $\Phi_m\in\mathcal{I}$ and $F_2^m-F_1^m=\Phi_m$. Then $\Phi_m$ converges uniformly to $F_2-F_1$ as $m\to\infty$, implying $F_1-F_2\in\widehat{\mathcal{I}}$.
\end{proof}

\begin{proposition}
  \label{trace property}
  For $A,B\in\Gau(S)$, we have
  \[
    \Tr(AB)\equiv\Tr(BA)\mod\widehat{\mathcal{I}}.
  \]
\end{proposition}

\begin{proof}
  By Corollary \ref{Gaussian trace-class}, $AB$ and $BA$ are piecewise trace-class operators, so their piecewise traces are defined. We consider a function $F\colon G\times G\to\R$ given by
  \[
    F(g,h)=\int_{gK\times hK}\mathrm{tr}(k_A(x,y)k_B(y,x))\vol(x)\vol(y).
  \]
  Since $A$ and $B$ have Gaussian propagation, there are $C_1,C_2>0$ such that for any $g,h\in G$,
  \begin{align*}
    |F(g,h)|&\le\int_{gK\times hK}|\mathrm{tr}(k_A(x,y)k_B(y,x))|\vol(x)\vol(y)\\
    &\le\int_{gK\times hK}C_1e^{-C_2d(x,y)^2}\vol(x)\vol(y).
  \end{align*}
  Then by \eqref{quasi-isometry}, the function $F$ satisfies \eqref{propagation F}. By Lemma \ref{trace integral}, we have
  \begin{align*}
    F_1(g)&=\sum_{h\in G}F(h,g)\\
    &=\int_{gK}\mathrm{tr}\left(\int_Mk_A(x,y)k_B(y,x)\vol(y)\right)\vol(x)\\
    &=\int_{gK}\mathrm{tr}(k_{AB}(x,x))\vol(x)\\
    &=\Tr(AB)(g)
  \end{align*}
  for any $g\in G$, implying $F_1=\Tr(AB)$. Quite similarly, we can see that $F_2=\Tr(BA)$. Thus by Lemma \ref{F_1=F_2}, we obtain $\Tr(AB)\equiv\Tr(BA)\mod\widehat{\mathcal{I}}$.
\end{proof}

%%%%% Subsection 4.2 %%%%%

\subsection{Functional calculus}

In the sequel, let $S\to M$ be a $G$-invariant Clifford bundle with Dirac operator $D$. Then $S$ has bounded geometry in the sense of \cite[p. 93]{R1}, and the Dirac operator $D$ is $G$-invariant. We consider a generalized Dirac operator
\[
  \widetilde{D}=D+A
\]
where $A$ is a self-adjoint endomorphism of $S$, not necessarily $G$-invariant. For the rest of this section, we assume that $\nabla^iA$ is bounded for $i=0,1,\ldots,\lceil\frac{n}{2}\rceil-1$. Let $C_c^\infty(S)$ denote the set of compactly supported smooth sections of $S$.

\begin{lemma}
  \label{D estimate}
  There is $C_k>0$ satisfying an ineuality
  \[
    \sum_{i=0}^k\|D^iu\|^2\le C_k\sum_{i=0}^k\|\widetilde{D}^iu\|^2
  \]
  for $k=0,1,\ldots\lceil\frac{n}{2}\rceil$ and any $u\in C^\infty_c(S)$.
\end{lemma}

\begin{proof}
  We induct on $k$. The $k=0$ case is trivial. We assume that the statement holds for $k\le m-1$. Then
  \begin{align*}
    \|D^mu\|&=\|D^{m-1}(\widetilde{D}-A)u\|\\
    &\le\|D^{m-1}\widetilde{D}u\|+\|D^{m-1}Au\|\\
    &\le C_1\left(\sum_{i=0}^{m-1}\|\widetilde{D}^{i+1}u\|^2\right)^\frac{1}{2}+\|D^{m-1}Au\|
  \end{align*}
  for some $C_1>0$ which is independent of the choice of $u$. We also have
  \[
    \|D^{m-1}Au\|^2\le C_2\sum_{i=0}^{m-1}\|\nabla^i(Au)\|^2\le C_3\sum_{i=0}^{m-1}\|\nabla^iu\|^2\le C_4\sum_{i=0}^{m-1}\|D^iu\|^2
  \]
  for some $C_2,C_3,C_4>0$ which are independent of the choice of $u$. Here, the first inequality follows as in \cite[page 73]{R4}, the second inequality follows from $\nabla(Au)=(\nabla A-A\nabla)u$ and the assumption that $\nabla^iA$ is bounded for $i=0,1,\ldots\lceil\frac{n}{2}\rceil-1$, and the third inequality follows from the elliptic estimate (cf. \cite[Proposition 5.16]{R4}). Then by the induction hypothesis, we obtain
  \[
    \|D^{m-1}Au\|^2\le C_5\sum_{i=0}^{m-1}\|\widetilde{D}^iu\|^2
  \]
  for some $C_5>0$ which is independent of the choice of $u$, and thus the statement follows.
\end{proof}

Let $\mathcal{G}$ denote the set of Schwartz-class functions $\phi\colon\R\to\R$ such that there are $C_1,C_2>0$ satisfying
\[
  \label{varphi}
  |\widehat{\phi}^{(i)}(x)|\le C_1e^{-C_2x^2}
\]
for $i=0,1,\ldots,2\lceil\frac{n}{2}\rceil$, where $\widehat{\phi}$ denotes the Fourier transform of $\phi$ and $f^{(i)}$ stands for the $i$-th derivative of a function $f\colon\R\to\R$.

\begin{proposition}
  The set of functions $\mathcal{G}$ is an algebra.
\end{proposition}

\begin{proof}
  The only nontrivial property is that $\mathcal{G}$ is closed under products. For Schwartz-class functions $\phi$ and $\psi$, we have $(\widehat{\phi\psi})^{(i)}=\widehat{\phi}^{(i)}*\widehat{\psi}$. Then since the convolution of Gaussian functions is a Gaussian function, $\mathcal{G}$ is closed under products.
\end{proof}

\begin{lemma}
  \label{essentially self-adjoint}
  A generalized Dirac operator $\widetilde{D}$ is essentially self-adjoint.
\end{lemma}

\begin{proof}
  By \cite[Theorem 1.17]{GL}, $D$ is essentially self-adjoint. Then as $A$ is self-adjoint and continuous, $\widetilde{D}$ is essentially self-adjoint too.
\end{proof}

Lemma \ref{essentially self-adjoint} enables us to perform functional calculus for a generalized Dirac operator $\widetilde{D}$. We remark that all results in \cite[Chapter 9]{R4} holds not only for the Dirac operator $D$ but also for a generalized Dirac opertor $\widetilde{D}$ as mentioned in \cite[p. 95]{R4}. In particular, we have:

\begin{lemma}
  [cf. {\cite[Proposition 7.20]{R4}}]
  \label{finite propagation}
  For any $u\in C_c^\infty(S)$, $e^{\sqrt{-1}t\widetilde{D}}u$ is supported within $N(\mathrm{supp}(u);|t|)$.
\end{lemma}

For $\phi\in\mathcal{G}$, we can define the operator $\phi(\widetilde{D})$ as in \cite[Chapter 9]{R4}, that is,
\[
  \langle u,\phi(\widetilde{D})v\rangle=\frac{1}{2\pi}\int_\R\widehat{\phi}(t)\langle u,e^{\sqrt{-1}t\widetilde{D}}v\rangle dt
\]
for $u,v\in C^\infty_c(S)$.

\begin{proposition}
  \label{Gaussian}
  The map
  \[
    \mathcal{G}\to\Gau(S),\quad\phi\mapsto\phi(\widetilde{D})
  \]
  is a well-defined homomorphism.
\end{proposition}

\begin{proof}
  Let $\widetilde{K}$ denote the $\epsilon$-neighborhood of $K$ for some $\epsilon>0$. We take any $g,h\in G$, and let $\delta$ be the distance between $g\widetilde{K}$ and $h\widetilde{K}$. We also take any $u\in L^2(S\vert_{g\widetilde{K}})$ and $v\in C^\infty_c(S)$ with $\mathrm{supp}(v)\subset h\widetilde{K}$. Let $\phi\in\mathcal{G}$. Then by Lemma \ref{finite propagation}, we have
  \begin{align*}
    &|\langle u,\widetilde{D}^i(\phi(\widetilde{D})\widetilde{D}^jv)\vert_{g\widetilde{K}}\rangle_{L^2(S\vert_{g\widetilde{K}})}|\\
    &=|\langle u,\widetilde{D}^i(\phi(\widetilde{D})\widetilde{D}^jv)\rangle_{L^2(S)}|\\
    &=\frac{1}{2\pi}\left|\int_\R\widehat{\phi}(t)\langle u,e^{\sqrt{-1}t\widetilde{D}}\widetilde{D}^{i+j}v\rangle_{L^2(S)}dt\right|\\
    &=\frac{1}{2\pi}\left|\int_\R\widehat{\phi}^{(i+j)}(t)\langle u,e^{\sqrt{-1}t\widetilde{D}}v\rangle_{L^2(S)}dt\right|\\
    &\le\frac{1}{2\pi}\int_{|t|\ge\delta}C_1e^{-C_2t^2}dt\|u\|_{L^2(S\vert_{g\widetilde{K}})}\|v\|_{L^2(S\vert_{h\widetilde{K}})}\\
    &\le C_3e^{-C_2\delta^2}\|u\|_{L^2(S\vert_{g\widetilde{K}})}\|v\|_{L^2(S\vert_{h\widetilde{K}})}
  \end{align*}
  for some $C_1,C_2,C_3>0$ which are independent of the choice of $u,v$, where $i,j=0,1,\ldots,\lceil\frac{n}{2}\rceil$. Hence as $u$ is an arbitrary element of $L^2(S\vert_{g\widetilde{K}})$, we get
  \begin{equation}
    \label{estimate for Sobolev}
    \|\widetilde{D}^i((\varphi(\widetilde{D})\widetilde{D}^jv)\vert_{g\widetilde{K}})\|_{L^2(S\vert_{g\widetilde{K}})}\le C_3e^{-C_2\delta^2}\|v\|_{L^2(S\vert_{h\widetilde{K}})}
  \end{equation}
  for $i,j=0,1,\ldots,\lceil\frac{n}{2}\rceil$. By Lemma \ref{D estimate}, we have the Sobolev inequality
  \[
    \sup_{x\in g\widetilde{K}}|(\varphi(\widetilde{D})\widetilde{D}^jv)(x)|\le C_4\left(\sum_{i=0}^{\lceil\frac{n}{2}\rceil}\|\widetilde{D}^i(\varphi(\widetilde{D})\widetilde{D}^jv)\vert_{g\widetilde{K}}\|_{L^2(S\vert_{g\widetilde{K}})}^2\right)^\frac{1}{2}
  \]
  for some $C_4>0$ which is independent of the choice of $v$ and $g$, where $j=0,1,\ldots,\lceil\frac{n}{2}\rceil$. Then by \eqref{estimate for Sobolev}, we obtain
  \begin{equation}
    \label{estimate at x}
    \sup_{x\in g\widetilde{K}}|(\varphi(\widetilde{D})\widetilde{D}^jv)(x)|\le\lceil\tfrac{n}{2}\rceil C_3C_4e^{-C_2\delta^2}\|v\|_{L^2(S\vert_{h\widetilde{K}})}
  \end{equation}
  for $j=0,1,\ldots,\lceil\frac{n}{2}\rceil$. Now we take any $x\in g\widetilde{K}$ and $u_x\in S_x$. Let $u=\mathbf{1}_{h\widetilde{K}}(y)k_{\phi(\widetilde{D})}(x,-)^*u_x\in L^2(S\vert_{h\widetilde{K}})$. It follows from \eqref{estimate at x} that
  \begin{align*}
    \langle\widetilde{D}^ju,v\rangle_{L^2(S\vert_{h\widetilde{K}})}&=\langle u,\widetilde{D}^jv\rangle_{L^2(S)}\\
    &=\left|\int_M\langle k_{\phi(\widetilde{D})}(x,y)^*u_x,v(y)\rangle\vol(y)\right|\\
    &=\left|\int_M\langle u_x,k_{\phi(\widetilde{D})}(x,y)v(y)\rangle\vol(y)\right|\\
    &=|\langle u_x,(\varphi(\widetilde{D})\widetilde{D}^jv)(x)\rangle|\\
    &\le \lceil\tfrac{n}{2}\rceil C_3C_4e^{-C_2\delta}|u_x|\|v\|_{L^2(S\vert_{h\widetilde{K}})}
  \end{align*}
  for $j=0,1,\ldots,\lceil\frac{n}{2}\rceil$. Then since $v$ is any element of $C_c^\infty(S\vert_{h\widetilde{K}})$, we get
  \[
    \|\widetilde{D}^ju\|_{L^2(S\vert_{h\widetilde{K}})}\le C_5e^{-C_2\delta^2}|u_x|.
  \]
  for some $C_5>0$ which is independent of the choice of $x$ and $u_x$, where $j=0,1,\ldots,\lceil\frac{n}{2}\rceil$. Thus by arguing as above, we obtain
  \[
    |k(x,y)^*u_x|\le C_7e^{-C_2\delta^2}|u_x|
  \]
  for some $C_7>0$ which is independent of the choice of $g$ and $h$, where $x\in g\widetilde{K}$ and $y\in h\widetilde{K}$. As $u_x$ is an arbitrary element of $S_x$, this implies that
  \[
    |k(x,y)|=|k(x,y)^*|\le C_7e^{-C_2\delta^2}
  \]
  for any $x\in g\widetilde{K}$ and $y\in h\widetilde{K}$. Thus by \eqref{quasi-isometry}, we obtain $\phi(\widetilde{D})\in\mathcal{G}$. Clearly, the map in the statement is a homomorphism, completing the proof.
\end{proof}

%%%%% Section 4 %%%%%

\section{Piecewise normalized Betti number}\label{Piecewise normalized Betti number}

In this section, we define piecewise normalized Betti numbers as an analogy to normalized Betti numbers defined by Roe \cite{R2,R3}, and prove its invariance under a certain deformation.

%%%%% Subsection 4.1 %%%%%

\subsection{Abstract differential}

Let $\Lambda^k=\Lambda^kT^*M$ for $k=0,1,\ldots,n$, and let $\Lambda=\bigoplus_{k=0}^n\Lambda^k$. Then $\Lambda$ is a $G$-invariant Clifford bundle with Dirac operator $D=d+d^*$. Roe \cite{R2} defined an abstract differential on $\Lambda$ and the normalized Betti number of it. Here, we define a slightly different operator adopted to our situation, but we ambiguously call it an abstract differential.

\begin{definition}
  \label{abstract differential}
  An operator $\tilde{d}\colon C_c^\infty(\Lambda^k)\to C_c^\infty(\Lambda^{k+1})$ is an \emph{abstract differential} if it satisfies that

  \begin{enumerate}
    \item $\tilde{d}^2=0$, and

    \item $\tilde{d}+\tilde{d}^*$ is a generalized Dirac operator, that is, $D+A$ for some self-adjoint endomorphism $A$ of $\Lambda$, such that $\nabla^iA$ is bounded for $i=0,1,\ldots\lceil\frac{n}{2}\rceil-1$.

    %\item If $\varphi\in\mathcal{S}$, then $\varphi(d+d^*)$ is a piecewise Hilbert-Schmidt operator.

    %\item If $\varphi\in\mathcal{G}$, then $\varphi(\tilde{d}+\tilde{d}^*)\in\mathcal{A}$.
  \end{enumerate}
\end{definition}

We give the most important example of an abstract differential.

\begin{example}
  \label{Witten deformation}
  Let $f\colon M\to\R$ be a function. The Witten deformation of the differential $d$ on $\Lambda$ is defined by $d_t=e^{-tf}de^{tf}$. Clearly, $d_t^2=0$ and
  \[
    d_t+d_t^*=D+tR
  \]
  where $R=(df\wedge)-(df\lrcorner)$. If $f$ is bounded, that is, $\nabla^if$ is bounded for $i=0,1,\ldots,\lceil\frac{n}{2}\rceil$, then $\nabla^iR$ is bounded for $i=0,1,\ldots,\lceil\frac{n}{2}\rceil-1$. Thus $d_t$ is an abstract differential.
\end{example}

We define piecewise normalized Betti numbers of an abstract differential by replacing trace  with piecewise trace in the definition of normalized Betti numbers \cite[p.766]{R3}. However, compactly supported nonzero functions do not belong to the algebra $\mathcal{G}$. Then, instead of showing that $\phi(\widetilde{D})$ is of piecewise trace-class for a compactly supported function $\phi\colon\R\to\R$, we extend the definition of piecewise trace.

\begin{definition}
  Let $A\colon L^2(\Lambda)\to L^2(\Lambda)$ be an operator with bounded kernel function. We define a bounded function
  \[
    \widetilde{\Tr}(A)\colon G\to\R,\quad g\mapsto\int_{gK}\mathrm{tr}(k_A(x,x))\vol(x).
  \]
\end{definition}

By Lemma \ref{trace integral}, we have:

\begin{proposition}
  If $A\colon L^2(\Lambda)\to L^2(\Lambda)$ is of piecewise trace-class, then
  \[
    \widetilde{\Tr}(A)=\Tr(A).
  \]
\end{proposition}

Let $\tilde{d}\colon\Lambda\to\Lambda$ be an abstract differential, and let $\widetilde{D}=\tilde{d}+\tilde{d}^*$.

\begin{lemma}
  \label{bounded kernel}
  For any Schwartz-class function $\phi\colon\R\to\R$, $\phi(\widetilde{D})$ has a bounded kernel function.
\end{lemma}

\begin{proof}
  By Lemma \ref{D estimate}, there is an embedding $\widetilde{W}^k\to W^k$, where $\widetilde{W}^k$ and $W^k$ denote the $k$-th Sobolev spaces with respect to $\widetilde{D}$ and $D$, respectively. Then the statement can be verified quite similarly to Proposition \ref{Gaussian}.
\end{proof}

Let $\mathcal{C}$ denote the set of compactly supported functions $\phi\colon\R_{\ge 0}\to\R$ such that $\phi(x^2)$ is smooth, $0\le\phi\le 1$ and $\phi(0)=1$.

\begin{definition}
  We define the \emph{piecewise normalized $k$-th Betti number} associated to an abstract differential $\tilde{d}$ as a bounded function
  \[
    b_k(\tilde{d})\colon G\to\R,\quad g\mapsto\inf_{\phi\in\mathcal{C}}\widetilde{\Tr}(\phi(\widetilde{D}^2\vert_{\Lambda^k}))(g).
  \]
\end{definition}

Clearly, $\widetilde{\Tr}(\phi(\widetilde{D}^2\vert_{\Lambda^k}))\ge 0$ for $\phi\in\mathcal{C}$, so the infimum in the above definition exists. We choose $C>0$ satisfying $|\Tr(e^{-\widetilde{D}^2})|\le C\mathbbm{1}$. We say that a subset $\mathcal{C}_0\subset\mathcal{C}$ is \emph{cofinal} if for any $\tilde{\phi}\in\mathcal{C}$ and $\epsilon>0$, there is $\phi\in\mathcal{C}_0$ satisfying
\begin{equation}
  \label{cofinal def}
  \tilde{\phi}(x^2)+C^{-1}\epsilon e^{-x^2}\ge\phi(x^2).
\end{equation}

\begin{lemma}
  \label{cofinal}
  If $\mathcal{C}_0$ is a cofinal subset of $\mathcal{C}$, then for each $g\in G$,
  \[
    b_k(\tilde{d})(g)=\inf_{\phi\in\mathcal{C}_0}\widetilde{\Tr}(\phi(\widetilde{D}^2\vert_{\Lambda_k}))(g).
  \]
\end{lemma}

\begin{proof}
  Take any $\tilde{\phi}\in\mathcal{C}$ and $\epsilon>0$. Then there is $\phi\in\mathcal{C}_0$ satisfying \eqref{cofinal def}, so we get
  \[
    \widetilde{\Tr}(\tilde{\phi}(\widetilde{D}^2\vert_{\Lambda_k}))+\epsilon\ge\widetilde{\Tr}(\phi(\widetilde{D}^2\vert_{\Lambda_k})).
  \]
  Thus as $\epsilon$ can be arbitrarily small, we obtain
  \[
    b_k(\tilde{d})(g)\ge\inf_{\phi\in\mathcal{C}_0}\widetilde{\Tr}(\phi(\widetilde{D}^2\vert_{\Lambda_k}))(g)
  \]
  for each $g\in G$. The reverse inequality is obvious by the definition of $b_k(\tilde{d})$, and therefore the statement follows.
\end{proof}

Any function on $[0,1]$ may be considered as a function on $\R$ by extending it by zero outside of $[0,1]$. Then we can define $\widetilde{\Tr}(\phi(\widetilde{D}^2\vert_{\Lambda^k}))$ for any $\phi\in C^\infty[0,1]$.

\begin{lemma}
  \label{Riesz}
  For any bounded linear function $\alpha\colon\ell^\infty(G)\to\R$, there is a signed regular Borel measure $\mu$ on $[0,1]$ such that for any $\phi\in C^\infty[0,1]$,
  \[
    \alpha(\widetilde{\Tr}(\phi(\widetilde{D}^2\vert_{\Lambda^k})))=\int_0^1\phi d\mu.
  \]
\end{lemma}

\begin{proof}
  There is $C_0>0$ such that
  \begin{equation}
    \label{estimate trace}
    |\widetilde{\Tr}(\phi(\widetilde{D}^2\vert_{\Lambda^k}))|\le\sup_{x\in[0,1]}|\phi(x)||\widetilde{\Tr}(\mathbf{1}_{[0,1]}(\widetilde{D}^2\vert_{\Lambda^k}))|\le C_0\sup_{x\in[0,1]}|\phi(x)|\mathbbm{1}
  \end{equation}
  for any $\phi\in C^\infty[0,1]$. Then the map
  \[
    C^\infty[0,1]\to\R,\quad\phi\mapsto\widetilde{\Tr}(\phi(\widetilde{D}^2\vert_{\Lambda^k}))
  \]
   is a bounded linear function. Thus the statement follows from the Riesz representation theorem \cite[6.19 Theorem]{Ru}.
\end{proof}

\begin{proposition}
  \label{Tr(e)}
  The piecewise trace $\Tr(e^{-t\widetilde{D}^2\vert_{\Lambda^k}})$ converges weakly to $b_k(\tilde{d})$ as $t\to\infty$.
\end{proposition}

\begin{proof}
  Clearly, there is $\phi\in\mathcal{C}$ satisfying $e^{-x^2}\ge\phi(x^2)$. Let $\mathcal{C}_0=\{\phi(tx)\mid t>0\}$. Then it is easy to see that $\mathcal{C}_0$ is a cofinal subset of $\mathcal{C}$. Hence by Lemma \ref{Riesz} and Lebesgue's monotone convergence theorem, $\widetilde{\Tr}(\phi(t\widetilde{D}^2\vert_{\Lambda^k}))$ converges weakly to $b_k(\tilde{d})$ as $t\to\infty$. Observe that for any $\epsilon,s>0$, there is $t>0$ satisfying
  \[
    \phi(sx^2)+C^{-1}\epsilon e^{-x^2}\ge e^{-tx^2}
  \]
  where $C>0$ is the constant in the definition of a cofinal subset of $\mathcal{C}$. Then
  \[
    \widetilde{\Tr}(\phi(t\widetilde{D}^2\vert_{\Lambda^k}))\le\Tr(e^{-t\widetilde{D}^2\vert_{\Lambda^k}})\le\widetilde{\Tr}(\phi(s\widetilde{D}^2\vert_{\Lambda^k}))+\epsilon\mathbbm{1}.
  \]
  Let $\alpha\colon\ell^\infty(G)\to\R$ be a positive bounded linear function. Then we have
  \[
    \alpha(\widetilde{\Tr}(\phi(t\widetilde{D}^2\vert_{\Lambda^k})))\le\alpha(\Tr(e^{-t\widetilde{D}^2\vert_{\Lambda^k}}))\le\alpha(\widetilde{\Tr}(\phi(s\widetilde{D}^2\vert_{\Lambda^k})))+\epsilon\alpha(\mathbbm{1}).
  \]
  Note that $\epsilon$ can be arbitrarily small, and accordingly, $s$ and $t$ can be arbitrarily large. Then we get
  \begin{equation}
    \label{weak conv}
    \lim_{t\to\infty}\alpha(\Tr(e^{-t\widetilde{D}^2\vert_{\Lambda^k}}))\to\alpha(b_k(\tilde{d})).
  \end{equation}
  As in \cite[Chapter IV, 8.16]{DS}, any bounded linear function on $\ell^\infty(G)$ is given by the integral with a bounded finitely additive signed measure, so by \cite[Chapter III, 1.8]{DS}, it admits the Jordan decomposition. Then \eqref{weak conv} holds for any bounded linear function $\alpha\colon\ell^\infty(G)\to\R$, completing the proof.
\end{proof}

%%%%% Subsection 4.2 %%%%%

\subsection{Invariance}

The rest of this section is devoted to prove the following theorem which is an analogy to \cite[(2.6) Theorem]{R3}.

\begin{theorem}
  \label{Betti number invariance}
  Let $F\colon \Lambda\to \Lambda$ be a bounded self-adjoint automorphism, and let $\tilde{d}_1,\tilde{d}_2$ be abstract differentials satisfying
  \[
    \tilde{d}_1F=F\tilde{d}_2.
  \]
  Then for $k=0,1,\ldots,n$, we have
  \[
    b_k(\tilde{d}_1)\approx b_k(\tilde{d}_2)\mod\overline{\mathcal{I}}.
  \]
\end{theorem}

The proof of this theorem is basically a modification of that of \cite[(2.6) Theorem]{R3}.

\begin{lemma}
  [cf. {\cite[Theorem 4.1]{R3}}]
  \label{E}
  Let $E\colon L^2(\Lambda^k)\to L^2(\Lambda^k)$ be a self-adjoint piecewise trace-class operator. Then
  \[
    \Tr(E)\ge-\rho_1(E-E^2).
  \]
\end{lemma}

\begin{proof}
  By Proposition \ref{HS norm}, $E^2$ is of piecewise trace-class, so $\rho_1(E-E^2)$ and the piecewise traces of $E^2$ and $E-E^2$ are defined. As $E$ is self-adjoint, $E^2$ is positive, so by Proposition \ref{trace positive}, $\Tr(E^2)\ge 0$. Then we get
  \[
    \Tr(E)=\Tr(E-E^2)+\Tr(E^2)\ge\Tr(E-E^2)\ge-\rho_1(E-E^2)
  \]
  and hence the statement follows.
\end{proof}

For positive integers $p$ and $q$, we set
\[
  P_p=e^{-p(\tilde{d}_1+\tilde{d}_1^*)^2\vert_{\Lambda^k}}\quad\text{and}\quad Q_q=F^{-1}e^{-q(\tilde{d}_2+\tilde{d}_2^*)^2\vert_{\Lambda^k}}F
\]
where $\tilde{d}_1,\tilde{d}_2$ and $F$ are as in Theorem \ref{Betti number invariance}.

\begin{lemma}
  \label{lim Q}
  The operator $Q_q$ is of piecewise trace-class such that
  \[
    \Tr(Q_q)=\Tr(FQ_qF^{-1}).
  \]
\end{lemma}

\begin{proof}
  Observe that
  \[
    k_{AF^{-1}}(x,y)=k_{A}(x,y)F^{-1}\quad\text{and}\quad k_{AF}(x,y)=k_{A}(x,y)F
  \]
  where $A=e^{-\frac{q}{2}(\tilde{d}_2+\tilde{d}_2^*)^2\vert_{\Lambda^k}}$. Then since $F$ is bounded and $A$ is piecewise Hilbert-Schmidt, the proof of Lemma \ref{HS norm integral} implies that $AF^{-1}$ and $AF$ are piecewise Hilbert-Schmidt, implying that $Q_q=(AF^{-1})^*(AF)$ is of piecewise trace-class. Moreover, we have
  \[
    \mathrm{tr}(k_{FQ_qF^{-1}}(x,y))=\mathrm{tr}(Fk_{Q_q}(x,y)F^{-1})=\mathrm{tr}(k_{Q_q}(x,y)).
  \]
  Thus by Lemma \ref{trace integral}, the statement follows.
\end{proof}

\begin{lemma}
  [cf. {\cite[(5.4) Corollary]{R3}}]
  \label{P^2-PQP}
  There is an inequality
  \[
    \rho_2(P_{2p}-P_pQ_qP_p)\le C\sqrt{\frac{q}{p-1}}\mathbbm{1}.
  \]
  for some $C>0$ which is independent of the choice of $p$ and $q$.
\end{lemma}

\begin{proof}
  By \cite[(5.3) Proposition]{R3}, we have
  \[
    \|P_{2p-2}-P_{p-1}Q_qP_{p-1}\|\le 4\|F\|\|F^{-1}\|\sqrt{\frac{q}{p-1}}
  \]
  so by Propositions \ref{HS norm}, we get
  \[
    \rho_2(P_{2p}-P_{p}Q_qP_{p})\le\|P_1\|\|P_{2p-2}-P_{p-1}Q_qP_{p-1}\|\rho_2(P_1)\le C\sqrt{\frac{q}{p-1}}\mathbbm{1}
  \]
  for some $C>0$ which is independent of the choice of $p$ and $q$.
\end{proof}

\begin{lemma}
  [cf. {\cite[(6.2) Proposition]{R3}}]
  \label{P-P^2}
  Let $f\colon\N\to\N$ be an increasing function. Then $\rho_1(P_p-P_{f(p)})$ and $\rho_1(Q_q-Q_{f(q)})$ converge weakly to zero as $p\to\infty$ and $q\to\infty$.
\end{lemma}

\begin{proof}
  By Lemma \ref{Tr(e)}, $\Tr(P_p)$ is weakly convergent as $p\to\infty$, so $\rho_1(P_p-P_{f(p)})=\Tr(P_p-P_{f(p)})$ converges weakly to zero as $p\to\infty$. By Lemma \ref{lim Q}, we also get the result for $\rho_1(Q_q-Q_{f(q)})$.
\end{proof}

\begin{lemma}
  \label{estimate rho}
  Let $A\colon L^2(\Lambda)\to L^2(\Lambda)$ be a bounded operator, and let $B\colon L^2(\Lambda)\to L^2(\Lambda)$ be a piecewise Hilbert-Schmidt operator. Then for $p+q\ge 1$,
  \[
    \rho_1(ABP_pQ_q)\le C\|A\|\rho_2(B)
  \]
  where $C>0$ is independent of the choice of $A,B,p,q$.
\end{lemma}

\begin{proof}
  By Propositions \ref{HS norm} and \ref{HS trace-class}, we have
\[
  \rho_1(ABP_pQ_q)\le\rho_2(AB)\rho_2(P_pQ_q)\le\|A\|\rho_2(B)\rho_2(P_pQ_q).
\]
If $q>0$, then by Proposition \ref{HS norm},
\[
  \rho_2(P_pQ_q)\le\|P_pQ_{q-1}\|\rho_2(Q_1)=\rho_2(Q_1).
\]
If $q=0$, then we can see quite similarly that $\rho_2(P_pQ_q)\le\rho_2(P_1)$. Since $\rho_2(P_1)$ and $\rho_2(Q_1)$ are bounded functions, the statement follows.
\end{proof}

\begin{lemma}
  \label{P vs Q}
  Let $\alpha\colon\ell^\infty(G)\to\R$ be a positive bounded linear function. Suppose that for $0<\epsilon<1$ and $i=0,1$,
  \begin{alignat*}{2}
    &\alpha(\rho_1(P_p-P_{2p+i}))<\epsilon\qquad&&\alpha(\rho_1(Q_q-Q_{2q+i}))<\epsilon\\
    &\alpha(\rho_1(P_p-P_{p+1}))<\epsilon&&\rho_2(P_{2p}-P_pQ_qP_p)<\epsilon\mathbbm{1}.
  \end{alignat*}
  Then there is an inequality
  \begin{align*}
    \alpha(\Tr(P_p))&<\alpha(\Tr(Q_q))+\alpha(\Tr([P_pQ_q,P_{p+1}]))\\
    &\quad+\alpha(\Tr([P_pQ_{q+1},Q_q]))+C\epsilon^\frac{1}{2}
  \end{align*}
  for some $C>0$ which is independent of the choice of $p,q,\epsilon$.
\end{lemma}

\begin{proof}
  By Lemma \ref{estimate rho}, we have
  \begin{align*}
    \Tr(P_p)&\le\rho_1(P_p-P_{2p+1})+\rho_1((P_{2p}-P_pQ_qP_{p})P_1)+\Tr(P_pQ_qP_{p+1})\\
    &\le\rho_1(P_p-P_{2p+1})+C_1\rho_2(P_{2p}-P_pQ_qP_p)+\Tr(P_pQ_qP_{p+1})
  \end{align*}
  for some $C_1>0$ which is independent of the choice of $p,q,\epsilon$. Then we get
  \[
    \alpha(\Tr(P_p))\le\alpha(\Tr(P_pQ_qP_{p+1}))+(1+C_1)\epsilon.
  \]
  By Lemma \ref{estimate rho}, we also have
  \begin{align*}
    \Tr(P_pQ_qP_{p+1})&\le\Tr(P_pQ_q)+\rho_1((P_{2p+1}-P_p)Q_q)+\Tr([P_pQ_q,P_{p+1}])\\
    &\le\Tr(P_pQ_q)+C_2\rho_2(P_{2p+1}-P_p)+\Tr([P_pQ_q,P_{p+1}])
  \end{align*}
  for some $C_2>0$ which is independent of the choice of $p,q,\epsilon$.
  Observe that by Proposition \ref{HS norm},
  \begin{equation}
    \label{P<e}
    \rho_2(P_{p_1}-P_{p_2})=\rho_1((P_{p_1}-P_{p_2})^2)^\frac{1}{2}\le\rho_1(P_{p_1}-P_{p_2})^\frac{1}{2}
  \end{equation}
  for $p_1<p_2$ because $(e^{-p_1x}-e^{-p_2x})^2\le e^{-p_1x}-e^{-p_2x}$. Then we get
  \[
    \alpha(\Tr(P_pQ_qP_{p+1}))\le\alpha(\Tr(P_pQ_q))+\alpha(\Tr([P_pQ_q,P_{p+1}]))+C_2\epsilon^\frac{1}{2}.
  \]
  Since $\rho_2(P_p)\le\|P_{p-1}\|\rho_2(P_1)=\rho_2(P_1)$ by Propositioin \ref{HS norm}, we have
  \begin{align*}
    \Tr(P_pQ_q)&\le\Tr(P_pQ_{2q+1})+\rho_1(P_p(Q_q-Q_{2q+1}))\\
    &\le\Tr(P_pQ_{2q+1})+\rho_2(P_p)\rho_2(Q_q-Q_{2q+1})\\
    &\le\Tr(P_pQ_{2q+1})+\rho_2(P_1)\rho_2(Q_q-Q_{2q+1})\\
    &=\Tr(Q_qP_pQ_{q+1})+\Tr([P_qQ_{q+1},Q_q])+\rho_2(P_1)\rho_2(Q_q-Q_{2q+1}).
  \end{align*}
  Quite similarly to \eqref{P<e}, we can see that
  \begin{equation}
    \label{Q<e}
    \rho_2(Q_{q_1}-Q_{q_2})\le\|F\|^2\|F^{-1}\|^2\rho_1(Q_{q_1}-Q_{q_2})^\frac{1}{2}
  \end{equation}
  for $q_1<q_2$. Then we get
  \[
    \alpha(\Tr(P_pQ_q))\le\alpha(\Tr(Q_qP_pQ_{q+1}))+\alpha(\Tr([P_qQ_{q+1},Q_q]))+C_3\epsilon^\frac{1}{2}
  \]
  for some $C_3>0$ which is independent of the choice of $p,q,\epsilon$. Summarizing, we obtain
  \begin{align*}
    \alpha(\Tr(P_p))&\le\alpha(\Tr(Q_qP_pQ_{q+1}))+\alpha(\Tr([P_pQ_q,P_{p+1}]))\\
    &\quad+\alpha(\Tr([P_pQ_{q+1},Q_q]))+C_4\epsilon^\frac{1}{2}
  \end{align*}
  for some $C_4>0$ which is independent of the choice of $p,q,\epsilon$. Now we set $E=Q_q-Q_qP_pQ_{q+1}$. Then
  \begin{align*}
    E-E^2&=(Q_q-Q_{2q})+(Q_{2q}-Q_q)P_pQ_{q+1}+Q_qP_pQ_1(Q_{2q}-Q_q)\\
    &\quad+Q_q(P_p-P_{2p})Q_{q+1}+Q_q(P_{2p}-P_pQ_qP_p)Q_{q+1}\\
    &\quad+Q_qP_p(Q_q-Q_{2q+1})P_pQ_{q+1}.
  \end{align*}
  and so by Lemma \ref{estimate rho} together with \eqref{P<e} and \eqref{Q<e},
  \[
    \alpha(\rho_1(E-E^2))\le C_5\epsilon^\frac{1}{2}
  \]
  for some $C_5>0$ which is independent of the choice of $p,q,\epsilon$. Hence by Lemma \ref{E}, we get
  \[
    \alpha(\Tr(Q_qP_pQ_{q+1}))=\alpha(\Tr(Q_q))-\alpha(\Tr(E))\le\alpha(\Tr(Q_q))+C_5\epsilon^\frac{1}{2}.
  \]
  Thus the statement follows.
\end{proof}

Now we are ready to prove Theorem \ref{Betti number invariance}.

\begin{proof}
  [Proof of Theorem \ref{Betti number invariance}]
  Let $\alpha\colon\ell^\infty(G)\to\R$ be any bounded positive linear function. By Lemmas \ref{P^2-PQP} and \ref{P-P^2}, there are integer sequences $p_1<p_2<\cdots$ and $q_1<q_2<\cdots$ such that
  \begin{alignat*}{2}
    &\alpha(\rho_1(P_{p_i}-P_{2p_i+j}))<\tfrac{1}{i}\qquad&&\alpha(\rho_1(Q_{q_i}-Q_{2q_i+j}))<\tfrac{1}{i}\\
    &\alpha(\rho_1(P_{p_i}-P_{p_i+1}))<\tfrac{1}{i}&&\rho_2(P_{2p_i}-P_{p_i}Q_{q_i}P_{p_i})<\tfrac{1}{i}\mathbbm{1}
  \end{alignat*}
  where $j=0,1$. By Lemmas \ref{Tr(e)} and \ref{estimate rho} together with \eqref{P<e} and \eqref{Q<e}, it is straightforward to see that $\alpha(\Tr([P_{p_i}Q_{q_i},P_{p_i+1}])$ and $\alpha(\Tr([P_{p_i}Q_{q_i+1},Q_{q_i}]))$ are Cauchy sequences. Then they converges as $i\to\infty$, and so by Proposition \ref{trace property} and the Jordan decomposition mentioned in the proof of Lemma \ref{Tr(e)}, as $i\to\infty$, $\Tr([P_{p_i}Q_{q_i+1},Q_{q_i}])$ and $\Tr([P_{p_i}Q_{q_i+1},Q_{q_i}])$ converge weakly to some elements of $\bar{\mathcal{I}}$ which are independent of the choice of $\alpha$. Thus by Lemmas \ref{Tr(e)}, \ref{lim Q} and \ref{P vs Q}, we get
  \[
    \alpha(b_k(\tilde{d}_1))\ge\alpha(b_k(\tilde{d}_2))+\alpha(\delta)
  \]
  for some $\delta\in\bar{\mathcal{I}}$ which is independent of the choice of $\alpha$. Thus since the substitution of an element of $G$ is a positive bounded linear function on $\ell^\infty(G)$, we obtain
  \[
    b_k(\tilde{d}_1)\ge b_k(\tilde{d}_2)\mod\bar{\mathcal{I}}.
  \]
  By symmetry, we have the reverse inequality, and therefore the proof is finished.
\end{proof}

%%%%% Section 5 %%%%%

\section{Morse inequality}\label{Morse inequalities}

In this section, we prove Theorem \ref{main} and its corollaries by collecting the results obtained in the previous sections. Then we pose some questions on the boundedness of a Morse function that we anticipate will be fruitful.

%%%%% Subsection 5.1 %%%%%

\subsection{Strongly uniform Morse function}

The key technical ingredient of Morse theory is the Morse lemma which states that near a critical point, a Morse function is identified with a nondegenerate quadratic function by choosing an appropriate coordinate. Such a coordinate is not generally a normal coordinate, which is inconvenient for our purpose. Then, instead of Morse lemma, we consider replacement of a given Morse function by a nice Morse function with the same critical points. We show that this is possible for a strongly uniform Morse function, and so we only need to prove Theorem \ref{main} for such a nice Morse function. In the sequel, let $L$ denote the $G$-invariant triangulation of $M$ which is used to construct a fundamental domain $K$.

\begin{definition}
  \label{strong uniformness}
  A bounded Morse function $f\colon M\to\R$ is called \emph{strongly uniform} if it is uniform for $\epsilon>0$ and satisfies the following conditions:
  \begin{enumerate}
    \item there is $\delta_1>0$ such that $|\lambda|\ge\delta_1$, where $\lambda$ ranges over all eigenvalues of the Hessian $\nabla^2f(p)$ for all $p\in\Crit(f)$;

    \item there is $\delta_2>0$ such that $|\nabla f|\ge\delta_2$ on $M-N(\Crit(f);\epsilon_0)$;

    \item the $\epsilon$-neighborhood of any critical point of $f$ is contained in some $n$-simplex of $L$.
  \end{enumerate}
\end{definition}

We set
\[
  \epsilon_0=\frac{\delta_2+8(\sup|\nabla^3f|)\epsilon^2}{\delta_1}.
\]
Since $\delta_2$ and $\epsilon$ can be arbitrarily small, we may assume that
\begin{equation}
  \label{assumptioin epsilon}
  \epsilon_0<\frac{\epsilon}{2}.
\end{equation}
Let $g$ denote the fixed $G$-invariant metric on $M$. We introduce a new $G$-invariant metric on $M$. Take any small $\epsilon>0$. We define a $G$-invariant metric $g(\epsilon)$ on $M$ such that $g(\epsilon)=g$ on $N(L_{n-1};\frac{\epsilon}{2})$ and $g(\epsilon)$ is flat on the complement of $N(L_{n-1};\epsilon)$, where $L_{n-1}$ denotes the $(n-1)$-skeleton of $L$. Indeed, such a $G$-invariant metric $g(\epsilon)$ exists. Let $f\colon M\to\R$ be a uniform Morse function with respect to $\epsilon>0$. Then for any critical point $p$ of $f$, $N(p;\epsilon)$ is in the complement of $N(L_{n-1};\epsilon)$, and so the metric $g(\epsilon)$ is flat on $N(p;\epsilon)$. We remark that all results in the previous sections hold for any $G$-invariant metric on $M$, hence for $g(\epsilon)$. As mentioned above, boundedness of a function is independent of the choice of a $G$-invariant metric on $M$. We will use the fixed metric $g$ for considering upper and lower bounds for a Morse function.

Let us deform a given strongly uniform Morse function on $M$ to the one which is quadratic near each critical point. Let $f\colon M\to\R$ be a uniform Morse function with respect to $\epsilon>0$, and consider the metric $g(\epsilon)$. Take any $p\in\Crit(f)$. Then there is a flat normal coordinate $(x_1,\ldots,x_n)$ around $p$ such that $(0,\ldots,0)$ corresponds to $p$. By rotating the flat normal coordinate if necessary, we have the Taylor expansion of $f$ at $p$ as
\begin{equation}
  \label{Taylor expansion}
  f(x)=f(p)+\frac{1}{2}(\lambda_1x_1^2+\cdots+\lambda_nx_n^2)+r_p(x)
\end{equation}
where $\lambda_1,\ldots,\lambda_n$ are eigenvalues of the Hessian $\nabla^2f(p)$ and $r_p(x)=O(d(x,p)^3)$ as $x\to p$.

\begin{lemma}
  \label{Taylor expansion lemma}
  If $f\colon M\to\R$ is a strongly uniform bounded Morse function with respect to $\epsilon>0$, then for any $p\in\Crit(f)$ and $x\in N(p;\epsilon_0)-N(p;\frac{\epsilon_0}{2})$, the Taylor expansion \eqref{Taylor expansion} satisfies
  \[
    ((\lambda_1x_1)^2+\cdots+(\lambda_nx_n)^2)^\frac{1}{2}-|\nabla r_p(x)|-\frac{3}{\epsilon}|r_p(x)|>\delta_2
  \]
  where $\epsilon_0$ and $\delta_2$ are as in Definition \ref{strong uniformness}.
\end{lemma}

\begin{proof}
  Let $C=\sup|\nabla^3f|$, and take any $x\in N(p;\epsilon)$. Then by the mean value theorem, we have
  \[
    |\nabla^2r_p(x)|\le|\nabla^3r_p(y)|\epsilon
  \]
  for some $y\in N(p;\epsilon)$ because $\nabla r_p(p)=\nabla^2r_p(p)=0$. Hence $|\nabla^2r_p(x)|\le C\epsilon$ as $\nabla^3r_p(x)=\nabla^3f(x)$ on $N(p;\epsilon)$. Quite similarly, we can get $|\nabla r_p(x)|\le C\epsilon^2$ and $|r_p(x)|\le C\epsilon^3$. Thus by \eqref{assumptioin epsilon}
  \[
    ((\lambda_1x_1)^2+\cdots+(\lambda_nx_n)^2)^\frac{1}{2}-|\nabla r_p(x)|-\frac{3}{\epsilon}|r_p(x)|\ge\delta_1|x|-4C\epsilon^2>\delta_2.
  \]
  Therefore the statement is proved.
\end{proof}

\begin{proposition}
  \label{nice Morse function}
  Let $f\colon M\to\R$ be a strongly uniform bounded Morse function. Then there is a strongly uniform bounded Morse function $\tilde{f}\colon M\to\R$ satisfying that

  \begin{enumerate}
    \item $\Crit_k(\tilde{f})=\Crit_k(f)$ for $k=0,1,\ldots,n$, and

    \item for any $p\in\Crit(f)$,
    \[
      \tilde{f}(x)=\tilde{f}(p)+\frac{1}{2}(\lambda_1x_1^2+\cdots+\lambda_nx_n^2)
    \]
    on $N(p;\frac{\epsilon_0}{2})$, where $\lambda_1,\ldots,\lambda_n$ and $(x_1,\ldots,x_n)$ are as in \eqref{Taylor expansion}.
  \end{enumerate}
\end{proposition}

\begin{proof}
  Let $\rho\colon\R\to\R$ be an increasing smooth function such that
  \[
    \rho(t)=
    \begin{cases}
      0&t\le\frac{\epsilon_0}{2}\\
      1&t\ge\epsilon_0.
    \end{cases}
    \quad\text{and}\quad 0\le\frac{d\rho}{dt}\le\frac{3}{\epsilon_0}.
  \]
  Now we define
  \[
    \tilde{f}(x)=f(p)+\frac{1}{2}(\lambda_1x_1^2+\cdots+\lambda_nx_n^2)+\rho(d(x,p))r_p(x)
  \]
  for $p\in\Crit(f)$ and $x\in N(p;\epsilon_0)$, and $\tilde{f}=f$ on $M-N(\Crit(f);\epsilon_0)$. Clearly, $\Crit_k(\tilde{f})=\Crit_k(f)$ for $k=0,1,\ldots,n$. By Lemma \ref{Taylor expansion lemma}, $|\nabla\tilde{f}|\ge\delta_2$ on $N(\Crit(f);\epsilon_0)-N(\Crit(f);\frac{\epsilon_0}{2})$, and it is straightforward to check that $|\nabla\tilde{f}|\ge\delta_2$ on $M-N(\Crit(f);\epsilon_0)$. Thus the statement follows.
\end{proof}

%%%%% Subsection 5.2 %%%%%

\subsection{Proof of Theorem \ref{main}}

Let $f\colon M\to\R$ be a bounded function, and for $t>0$, let $d_t=e^{-tf}de^{tf}$ for the differential $d\colon C^\infty_c(\Lambda)\to C^\infty_c(\Lambda)$. Then as in Example \ref{Witten deformation}, $d_t$ is an abstract differential. Let $D_t=d_t+d_t^*$.

\begin{proposition}
  \label{inequality general}
  Let $\phi\colon\R\to\R$ be a function such that $\phi(x^2)\in\mathcal{G}$. and let $\psi(x)=x\phi(x)^2$. Then for $t>0$, we have
  \[
    \sum_{i=0}^k(-1)^{k-i}\Tr(\psi(D_t^2\vert_{\Lambda^i}))\ge 0\mod\widehat{\mathcal{I}}.
  \]
  for $k=0,1,\ldots,n-1$ and
  \[
    \sum_{i=0}^n(-1)^{n-i}\Tr(\psi(D_t^2\vert_{\Lambda_i}))\equiv 0\mod\widehat{\mathcal{I}}.
  \]
\end{proposition}

\begin{proof}
  Quite similarly to the proof of Proposition \ref{algebra}, we can show that the operator $d_t\phi(D_t^2\vert_{\Lambda^i})$ is piecewise Hilbert-Schmidt, so by Proposition \ref{algebra}, $d_t^*\phi(D_t^2\vert_{\Lambda^i})=(d_t\phi(D_t^2\vert_{\Lambda^i}))^*$ is piecewise Hilbert-Schmidt too. Then by Proposition \ref{trace property}, we have
  \begin{align*}
    \Tr(d_td_t^*\phi(D_t^2\vert_{\Lambda^i})^2)&=\Tr(d_t\phi(D_t^2\vert_{\Lambda^{i-1}})d_t^*\phi(D_t^2\vert_{\Lambda^i}))\\
    &\equiv\Tr(d_t^*\phi(D_t^2\vert_{\Lambda^i})d_t\phi(D_t^2\vert_{\Lambda^{i-1}}))\mod\widehat{\mathcal{I}}\\
    &=\Tr(d_t^*d_t\phi(D_t^2\vert_{\Lambda^{i-1}})^2).
  \end{align*}
  Hence we get
  \begin{align*}
    &\sum_{i=0}^k(-1)^{k-i}\Tr(\psi(D_t^2\vert_{\Lambda_i}))\\
    &=\sum_{i=0}^k(-1)^{k-i}(\Tr(d_td_t^*\phi(D_t^2\vert_{\Lambda^i})^2)+\Tr(d_t^*d_t\phi(D_t^2\vert_{\Lambda^i})^2))\\
    &\equiv\sum_{i=0}^k(-1)^{k-i}(\Tr(d_t^*d_t\phi(D_t^2\vert_{\Lambda^{i-1}})^2)+\Tr(d_t^*d_t\phi(D_t^2\vert_{\Lambda^i})^2))\mod\widehat{\mathcal{I}}\\
    &=\Tr(d_t^*d_t\phi(D_t^2\vert_{\Lambda^k})^2).
  \end{align*}
  Since $d_t^*d_t\phi(D_t^2\vert_{\Lambda^k})^2)=(d\phi(D_t^2\vert_{\Lambda^k}))(d\phi(D_t^2\vert_{\Lambda^k}))^*$, the operator $d_t^*d_t\phi(D_t^2\vert_{\Lambda^k})^2$ is positive, hence by Proposition \ref{trace positive}, $\Tr(d_t^*d_t\phi(D_t^2\vert_{\Lambda^k})^2)\ge 0$. Clearly, $d_t^*d_t\phi(D_t^2\vert_{\Lambda^n})^2=0$, and thus the proof is finished.
\end{proof}

\begin{corollary}
  \label{inequality betti}
  Let $\phi$ and $\psi$ be as in Proposition \ref{inequality general}. If $\psi(x^2)\ge e^{-x^2}$, then for any $t>0$, we have
  \[
    \sum_{i=0}^k(-1)^{k-i}\Tr(\psi(D_t^2\vert_{\Lambda^i}))\ge\sum_{i=0}^k(-1)^{k-i}b_i^{(2)}\mod\bar{\mathcal{I}}
  \]
  for $k=0,1,\ldots,n-1$, and
  \[
    \sum_{i=0}^n(-1)^{i}\Tr(\psi(D_t^2\vert_{\Lambda^i}))\approx(-1)^n\chi(M/G)\mod\bar{\mathcal{I}}.
  \]
\end{corollary}

\begin{proof}
  By Proposition \ref{inequality general}, for any $s>2$, we have
  \[
    \sum_{i=0}^k(-1)^{k-i}\Tr(\psi(D_t^2\vert_{\Lambda^i}))\ge\sum_{i=0}^k(-1)^{k-i}\Tr(e^{-sD_t^2\vert_{\Lambda^i}})\mod\widehat{\mathcal{I}}.
  \]
  Then by Lemma \ref{Tr(e)}, we take $s\to\infty$ to get
  \[
    \sum_{i=0}^k(-1)^{k-i}\Tr(\psi(D_t^2\vert_{\Lambda^i}))\ge\sum_{i=0}^k(-1)^{k-i}b_i(d_t)\mod\bar{\mathcal{I}}.
  \]
  Now we can apply Theorem \ref{Betti number invariance} to $F=e^{-tf},\tilde{d}_1=d,\,\tilde{d}_2=d_t$, so that $b_i(d_t)\approx b_i(d)=b_i^{(2)}\mod\bar{\mathcal{I}}$. Thus the inequalities in the statement follow. Quite similarly, we can get
  \[
    \sum_{i=0}^n(-1)^{i}\Tr(\psi(D_t^2\vert_{\Lambda^i}))\approx\sum_{i=0}^n(-1)^{n-i}b_i^{(2)}\mod\bar{\mathcal{I}}.
  \]
  Thus the proof is finished by the fact that $\sum_{i=0}^n(-1)^ib_i^{(2)}=\chi(M/G)$.
\end{proof}

Now we suppose that $f\colon M\to\R$ is a strongly uniform bounded Morse function, and let $\tilde{f}\colon M\to\R$ be as in Proposition \ref{nice Morse function}. We set $d_t=e^{-t\tilde{f}}de^{t\tilde{f}}$ and $D_t=d_t+d_t^*$.

\begin{lemma}
  \label{complement}
  Let $\phi\colon\R\to\R$ be a smooth function such that $\widehat{\phi}$ is supported within $(-\frac{\epsilon_0}{4},\frac{\epsilon_0}{4})$. Then on the complement of $N(\Crit(\tilde{f});\frac{\epsilon_0}{2})$, $\phi(D_t^2)$ converges uniformly to zero as $t\to\infty$.
\end{lemma}

\begin{proof}
  The proof is a modification of that for \cite[Lemma 14.6]{R4}. It is easy to see that $|\nabla\tilde{f}|\ge\delta$ for some $\delta>0$ on the complement of $N(\Crit(\tilde{f});\frac{\epsilon_0}{2})$. Let $u$ be any differential form compactly supported within the complement of $N(\Crit(\tilde{f});\frac{\epsilon_0}{2})$. Then as in the proof of \cite[Lemma 14.6]{R4}, there is an inequality
  \begin{equation}
    \label{D_t estimate}
    \|\phi(D_t^2)u\|\le q(t)\|u\|
  \end{equation}
  where $q(t)$ decays rapidly as $t\to\infty$ and is independent of the choice of $u$. As mentioned in the proof of \cite[(2.8) Proposition]{R1}, the Sobolev inequality
  \[
    \|u\|^2_{L^\infty}\le C_1\sum_{i=0}^{\lceil\frac{n}{2}\rceil}\|D^iu\|^2
  \]
  holds, where $C_1>0$ is independent of the choice of $u$. Then by Lemma \ref{D estimate}, the Sobolev inequality
  \[
    \|u\|^2_{L^\infty}\le C_2(t)\sum_{i=0}^{\lceil\frac{n}{2}\rceil}\|D^i_tu\|^2
  \]
  also holds, where $C_2(t)>0$ is independent of the choice of $u$ (but possibly dependent on $t$). On the other hand, since $D_t^2=D^2+t(\nabla^2\tilde{f})+t^2|\nabla\tilde{f}|^2$ and we are assuming that $\nabla\tilde{f}$ and $\nabla^2\tilde{f}$ are bounded, we have
  \[
    \|(1+D_t^2)^{-1}u\|_{k+2}\le p_k(t)\|u\|_k
  \]
  for some polynomial $p_k(t)$ which is independent of the choice of $u$, where $\|\cdot\|_k$ denotes the $k$-th Sobolev norm with respect to $D$. Then we can argue as in \cite[Proof of Lemma 14.6]{R4} to improve \eqref{D_t estimate} as
  \[
    \|\phi(D_t^2)u\|_{L^\infty}\le p(t)\|u\|_{L^1}
  \]
  where $p(t)$ decays rapidly as $t\to\infty$ and is independent of the choice of $u$. Thus the statement follows.
\end{proof}

Let $\tilde{\rho}\colon\R\to\R$ be a smooth function supported within $[-\frac{\epsilon_0}{2},\frac{\epsilon_0}{2}]$ such that $\tilde{\rho}(0)=1$. We define a smooth function $\rho\colon M\to\R$ by
\[
  \rho(x)=\sum_{p\in\Crit(f)}\tilde{\rho}(d(x,p)).
\]

\begin{lemma}
  \label{critical point}
  Let $\phi$ be as in Proposition \ref{complement}. Then  $\Tr(\rho\phi(D_t^2\vert_{\Lambda_k}))$ converges uniformly to $c_k$ as $t\to\infty$.
\end{lemma}

\begin{proof}
  Take any $p\in\Crit_k(\tilde{f})$. Then $p\in g(\mathrm{Int}(K))$ for some $g\in G$, so $g$ is unique. Let $\rho_p(x)=\tilde{\rho}(d(x,p))$. Recall that we are working with the metric $g(\epsilon)$ on $M$, which is flat on $N(p;\epsilon)$ for each $p\in\Crit(\tilde{f})$. Let $\lambda_1,\ldots,\lambda_n$ be the eigenvalues of the Hessian $\nabla^2\tilde{f}(p)$, and let
  \[
    \lambda(a,b)=\sum_{i=0}^n(|\lambda_i|(1+2a_i)+\lambda_ib_i)
  \]
  for $a=\{a_1,\ldots,a_n\}$ with nonnegative integers $a_i$ and $b=\{b_1,\ldots,b_n\}$ with $b_i=\pm 1$ such that exactly $k$ of $b_i$ are $+1$. Then as in the proof of \cite[Lemma 14.11]{R4}, $\phi(t\lambda(a,b)^\frac{1}{2})$ is a spectrum of $\phi(D_t^2\vert_{\Lambda^k})$ such that
  \[
    \Tr(\rho_p\phi(D_t^2\vert_{\Lambda^k}))(g)=\sum_{a,b}\phi(t\lambda(a,b)^\frac{1}{2})\langle\rho_pe(a,b),e(a,b)\rangle
  \]
  where $e(a,b)$ is a normalized eigenvector corresponding to $\lambda(a,b)$. Thus since $\phi$ is rapidly decaying, we get
  \[
    \lim_{t\to\infty}\sum_{a,b}\phi(t\lambda(a_i,b_i)^\frac{1}{2})=
    \begin{cases}
      1&p \text{ has index }k\\
      0&\text{otherwise}
    \end{cases}
  \]
  where only $\lambda(a,b)=0$ contributes to the first equality. As in the proof of \cite[Lemma 14.11]{R4}, if $\lambda(a,b)=0$, then
  \[
    e(a,b)=(\pi^{-\frac{n}{4}}t^\frac{n}{2}(\lambda_1\cdots\lambda_n)^\frac{n}{2})e^{-t\frac{1}{2}(\lambda_1x_1^2+\cdots+\lambda_nx_n^2)}dx_1\cdots dx_n
  \]
  hence $\langle\rho_pe(a,b),e(a,b)\rangle\to$ as $t\to 0$. On the other hand, we have
  \[
    \Tr(\rho_p\varphi(D_t^2\vert_{\Lambda^k}))(h)=0
  \]
  for $h\ne g$. Thus $\Tr(\rho_p\phi(D_t^2\vert_{\Lambda_k}))$ converges uniformly to a function $c_k^p\colon G\to\R$ given by
  \[
    c_k^p(h)=
    \begin{cases}
      1&p\text{ has index }k\text{ and }h=g\\
      0&\text{otherwise}.
    \end{cases}
  \]
  Since the Morse function $f$ is strongly uniform, we have
  \[
    \inf_{p\in\Crit(f)}|\lambda|=\delta_1>0
  \]
  where $\lambda$ ranges over all eigenvalues of the Hessian $\nabla^2\tilde{f}(p)$. Then the uniformness of the above convergence is also uniform as $p$ ranges over all $p\in\Crit(\tilde{f})$. Note that $\sum_{p\in\Crit(\tilde{f})}c_k^p=c_k$ and
  \[
    \Tr(\rho\varphi(D_t^2\vert_{\Lambda^k}))=\sum_{p\in\Crit(f)}\Tr(\rho_p\varphi(D_t^2\vert_{\Lambda^k}))
  \]
  where $\mathrm{supp}(\rho_p)\cap\mathrm{supp}(\rho_q)=\emptyset$ for $p\ne q\in\Crit(f)$. Thus the statement follows.
\end{proof}

Now we are ready to prove Theorem \ref{main}

\begin{proof}
  [Proof of Theorem \ref{main}]
  Combine Corollary \ref{inequality betti} and Lemmas \ref{complement} and \ref{critical point}.
\end{proof}

We prove Corollaries \ref{mean value} and \ref{betti critical points}.

\begin{proof}
  [Proof of Corollary \ref{mean value}]
  As in the proof of Proposition \ref{amenability}, we have $\mu(\bar{\mathcal{I}})=0$. Then the statement follows from Theorem \ref{main} and the positivity of $\mu$.
\end{proof}

To prove Corollary \ref{betti critical points}, we need the following proposition \cite[Proposition 2.3]{KKT}.

\begin{proposition}
  \label{l(G)}
  Let $G$ be a finitely generated infinite group. If $\phi\in\ell^\infty(G)$ satisfies $\phi(g)=0$ for all but finitely many $g\in G$, then $\phi\equiv 0\mod\mathcal{I}$.
\end{proposition}

\begin{proof}
  [Proof of Corollary \ref{betti critical points}]
  We can easily deduce from Theorem \ref{main} the weak Morse inequality
  \[
    c_k\ge b_k^{(2)}\mathbbm{1}\mod\bar{\mathcal{I}}.
  \]
  Then as in the proof of Corollary \ref{mean value}, we get $\mu(c_k)\ge b_k^{(2)}$ for any $G$-invariant mean $\mu\colon\ell^\infty(G)\to\R$. Hence by Proposition \ref{l(G)}, $c_k(g)\ne 0$ for infinitely many $g\in G$.
\end{proof}

Novikov and Shubin \cite{NS} proved Morse inequalities for a Morse function on $M/G$ and the $L^2$-Betti numbers of $M$. We recover this result when $G$ is amenable. For a Morse function $\bar{f}\colon M/G\to\R$, let $\bar{c}_k$ denote the number of critical points of $\bar{f}$ having index $k$.

\begin{corollary}
  \label{Novikov-Shubin}
  If $G$ is amenable and $\bar{f}\colon M/G\to\R$ is a Morse function, then we have
  \[
    \bar{c}_k-\bar{c}_{k-1}+\cdots+(-1)^k\bar{c}_0\ge b_k^{(2)}-b_{k-1}^{(2)}+\cdots+(-1)^kb_0^{(2)}
  \]
  for $k=0,1,\ldots,n-1$, and
  \[
    \bar{c}_n-\bar{c}_{n-1}+\cdots+(-1)^n\bar{c}_0=(-1)^n\chi(M/G).
  \]
\end{corollary}

\begin{proof}
  Let $f\colon M\to\R$ be the lift of $\bar{f}$. Then as $M/G$ is compact, $f$ is a bounded Morse function. Moreover, by deforming a fundamental domain $K$ slightly if necessary, we may assume that $f$ is strongly uniform. Since $f$ is $G$-invariant, we have $c_k=\bar{c}_k\mathbbm{1}$ for $k=1,2,\ldots,n$. Then the statement follows from Corollary \ref{mean value}.
\end{proof}

%%%%% Section 5 %%%%%

\section{Questions}\label{Questions}

In this section, we pose questions concerning the conditions on a Morse function. boundedness and uniformness of a Morse function is irrelevant to the $G$-action on $M$, and the strong uniformness of a Morse functions also irrelevant to the $G$-action on $M$ except for the third condition in Definition \ref{strong uniformness}. However, the relation of the third condition in Definition \ref{strong uniformness} and the $G$-action on $M$ is insignificant, and the authors believe that it is possible to drop this condition to prove Theorem \ref{main}. Then we pose:

\begin{question}
  Is it possible to prove Theorem \ref{main} without the third condition in Definition \ref{strong uniformness}?
\end{question}

As mentioned in Section \ref{Introduction}, the boundedness of the derivative $\nabla f$ of a Morse functioin $f\colon M\to\R$ is an essential assumption, and the boundedness of $\nabla^if$ for $i\ge 2$ and $f$ itself are assumed by technical reasons. Here, we consider the boundedness of $f$ itself. It is used to apply Theorem \ref{Betti number invariance} to the $F=e^{-tf}$ case. However, the normalized Betti number $b_k(d_t)$ for the Witten defrmation $d_t=e^{-tf}de^{tf}$ is defined in terms of the derivatives of $f$, so its properties may not depend on the boundedness of $f$.  Then we ask:

\begin{question}
  Does Theorem \ref{main} hold for an unbounded Morse function?
\end{question}

Corollary \ref{betti critical points} has its own interest, and so we also ask:

\begin{question}
  Does Corollary \ref{betti critical points} hold for an unbounded Morse functioin?
\end{question}

Observe that Corollary \ref{betti critical points} makes sense for any infinite group $G$, though Theorem \ref{main} makes sense only for an amenable group $G$ as in Proposition \ref{amenability}. Then we can further ask whether or not Corollary \ref{betti critical points} holds for an unbounded Morse function and a nonamenable group. However, we can get a negative answer to this further question as follows. Suppose that $M/G$ has nonpositive sectional curvature, and that $M$ is simply-connected. In this case, the group $G$ is nonamenable. Fix a basepoint $p$ of $M$. Then by the Cartan-Hadamard theorem, for each point $x\in M$, there is a unique geodesic $\gamma_x\colon[0,t_x]\to M$ from $p$ to $x$, where $t_x=d(p,x)$. Observe that for any small $\epsilon>0$, $|\frac{d\gamma_x}{dt}|=1$ on $[\epsilon,t_x]$. Then a function $f\colon M\to\R$ which is a small perturbation of a function
\[
  M\to\R,\quad x\mapsto d(p,x)
\]
has finitely many nondegenerate critical points near $p$. Namely, $f$ is a Morse function on $M$ with finitely many critical points. Thus if $G$ is infinite, e.g. $M/G$ is a surface of genus $\ge 2$, this example is a negative answer to the above further question.

%Next, we consider the boundedness of $\nabla^if$ for $i\ge 2$. There are two places where we use this condition; the first is Lemma \ref{D estimate}, and the second is Lemma \ref{bounded kernel}. In the first case, we only need the boundedness of $\nabla^if$ for $i\le\lceil\frac{n}{2}\rceil$. In the second case, the boundedness is used to apply \cite[(2.9) Proposition]{R1} which guarantees the boundedness of the kernel function and all of its derivatives. However, we do not need the boundedness of the derivatives of the kernel function, so we may weaken the boundedness assumption.

%\begin{question}
%  Is there a function $a(n)$ such that Theorem \ref{main} holds for a Morse function $f\colon M\to\R$ with bounded $\nabla^if$ for $i\le a(n)$?
%\end{question}


\begin{thebibliography}{99}
  \bibitem{A} M. Atiyah, Elliptic operators, discrete groups and von Neumann algebras, Ast\'erisque \textbf{32} (1976), 43-72.

  \bibitem{BW} J. Block and S. Weinberger, Aperiodic tilings, positive scalar curvature, and amenability of spaces, J. Amer. Math. Soc. \textbf{5} (1992), no. 4, 907-918.

  \bibitem{BNW} J. Brodzki, G.A. Niblo, and N. Wright, Pairings, duality, amenability and bounded cohomology, J. Eur. Math. Soc. \textbf{14} (2012), 1513-1518.

  \bibitem{DY} X. Dai and J. Yan, Witten deformation for noncompact manifolds with bounded geometry, J. Inst. Math. Jussieu \textbf{22} (2023), no. 2, 643-680.

  \bibitem{DS} N. Dunford and J.T. Schwartz, Linear Operators, Part I: General Theory, Wiley, New York, 1958.

  \bibitem{G} M. Gromov, Asymptotic invariants of infinite groups, in Geometric group theory, vol. 2 (Sussex, 1991), 1-295, London Math. Soc. Lecture Note Ser. \textbf{182}, Cambridge Univ. Press, Cambridge.

  \bibitem{GL} M. Gromov and H.B. Lawson, Jr., Positive scalar curvature and the Dirac operator on complete Riemannian manifolds, Publ. Math. I.H.E.S. \textbf{58} (1983), 83-196.

  \bibitem{H} M.W. Hirsch, On imbedding differentiable manifolds in euclidean space, Ann. of Math. (2) \textbf{73} (1961), 566-571.

  \bibitem{KKT1} T. Kato, D. Kishimoto, and M. Tsutaya, Homotopy type of the space of finite propagation unitary operators on $\Z$, Homology Homotopy Appl. \textbf{25} (2023), no. 1, 375-400.

  \bibitem{KKT2} T. Kato, D. Kishimoto, and M. Tsutaya, Homotopy type of the unitary group of the uniform Roe algebra on $\Z^n$, J. Topol. Anal. \textbf{15} (2023), no. 2, 495-512.

  \bibitem{KKT3} T. Kato, D. Kishimoto, and M. Tsutaya, Hilbert bundles with ends, J. Topol. Anal. \textbf{16} (2024), no. 2, 291-322.

  \bibitem{KKT} T. Kato, D. Kishimoto, and M. Tsutaya, Vector fields on non-compact manifolds, accepted by Algebr. Geom. Topol.

  \bibitem{L}W. L\"{u}ck, $L^2$-invariants: Theory and Applications to Geometry and $K$-theory, Ergebnisse der Mathematik und ihrer Grenzgebiete, 3. Folge Bd. \textbf{44}, Springer-Verlag, Berlin, 2002.

  %\bibitem{MM} X. Ma and G. Marinescu, Holomorphic Morse inequalities and Bergman kernels, Prog. in Math. \textbf{254}, Birkh\"{a}user Verlag, Basel, 2007.

  \bibitem{M} F. Milizia, Bounded differential forms and coinvariants of bounded functions, arVix:2311.07731.

  \bibitem{Mi} J. Milnor, A note on curvature and fundamental group, J. Differential Geometry \textbf{2} (1968) 1-7.

  \bibitem{NS} S.P. Novikov and M.A. Shubin, Morse inequalities and von Neumann $\mathrm{II}_1$-factors, Doklady Akad. Nauk SSSR \textbf{289} (1986), 289-292.

  \bibitem{P} P. Pansu, Cohomologie $L^p$: invariance sous quasiisom\'etries, preprint, 1995.

  \bibitem{R1} J. Roe, An index theorem on open manifolds. I, J. Differential Geom. \textbf{27} (1988), no. 1, 87-113.

  \bibitem{R2} J. Roe, An index theorem on open manifolds. II, J. Differential Geom. \textbf{27} (1988), no. 1, 115-136.

  \bibitem{R3} J. Roe, On the quasi-isometry invariance of $L^2$ Betti numnbers, Duke Math. J. \textbf{59} (1989), no. 3, 765-783.

  \bibitem{R4} J. Roe, Elliptic operators, Topology and Asymptotic Methods, Pitman Res. Notes Math. Ser. \textbf{179}, Longman Scientific \& Technical, Harlow; copublished in the United States with John Wiley \& Sons, Inc., New York, 1988.

  \bibitem{Ru} W. Rudin, Real and Complex Analysis, McGraw-Hill, 1966.
\end{thebibliography}
\end{document}